\documentclass[final,1p,times,authoryear]{elsarticle}

\usepackage[T1]{fontenc}
\usepackage[utf8]{inputenc}
\usepackage[english]{babel}
\usepackage{amsmath,amssymb,amsfonts,amsthm,mathtools}
\usepackage{microtype}
\usepackage{booktabs}
\usepackage{enumitem}
\usepackage[bookmarks=false,hidelinks]{hyperref}

\allowdisplaybreaks

\newcommand{\sinc}{\operatorname{sinc}}

\newcommand{\C}{\mathbb{C}}
\newcommand{\Z}{\mathbb{Z}}
\newcommand{\R}{\mathbb{R}}
\newcommand{\e}{\mathrm e}
\newcommand{\dd}{\,\mathrm d}
\newcommand{\Fcal}{\mathcal F}
\newcommand{\ip}[2]{\left\langle #1,#2\right\rangle}
\newcommand{\norm}[1]{\left\lVert #1\right\rVert}

\newtheorem{theorem}{Theorem}[section]
\newtheorem{lemma}[theorem]{Lemma}
\newtheorem{proposition}[theorem]{Proposition}
\newtheorem{corollary}[theorem]{Corollary}
\theoremstyle{definition}

\newtheorem{example}[theorem]{Example}
\theoremstyle{remark}
\newtheorem{remark}[theorem]{Remark}
\newtheorem{conjecture}[theorem]{Conjecture}

\begin{document}

\begin{frontmatter}

\title{Perturbed Exponential Systems, Pattern-Dependent Shifted-Sinc Operators, and Matrix Symbols on Finite Unions of Intervals}

\author[second]{Pierluigi Vellucci}
\author[third]{Gohin Shaikh Samad}

\affiliation[second]{organization={Department of Economics, Roma Tre University}}
\affiliation[third]{organization={Department of Mathematical Science, McNeese State University}}

\begin{abstract}
We study perturbed exponential systems and unions of shifted frequency lattices through shifted-sinc operators and finite-dimensional fiberizations. For the auxiliary binary local-shift model, we identify the exact pattern-dependent invariant
$$
\rho_{\delta,s}=\|P_-T_{2\delta}P_+\|,
$$
which gives the optimal squared synthesis bounds $1\pm\rho_{\delta,s}$. For periodic patterns, the synthesis operator reduces exactly to a single constant Vandermonde-type matrix, yielding finite-dimensional basis criteria and optimal bounds. We further prove that the worst binary lower bound is the infimum of its periodic counterparts; at the critical value $\delta=1/4$, a step pattern is governed by a half-shifted Hilbert matrix and has vanishing lower bound, although every fixed periodic pattern remains invertible.

We also combine exact bounds for structured patterns with a support-independent perturbation estimate, obtaining explicit stability neighborhoods. For finite unions of bounded intervals and finitely many shifted lattices, folding modulo one yields finitely many rectangular phase matrices whose singular values and ranks give the exact bounds and complete Riesz-sequence, frame, and Riesz-basis criteria. This framework also yields finite Hadamard constraints for periodic spectra and an exact reduction of rationally commensurate multirate sampling.
\end{abstract}

\begin{keyword}
perturbed exponentials \sep shifted-sinc operators \sep periodic perturbations \sep Riesz bases \sep Gram operators \sep matrix symbols \sep finite unions of intervals \sep variable-dimensional fibers
\end{keyword}

\end{frontmatter}

\section{Introduction}\label{sec:intro}

Perturbations of the standard exponential system naturally produce sinc kernels and discrete operators whose shifts depend on the frequency index. Closely related matrix structures also arise when exponential systems are restricted to disconnected spatial domains. A common shifted-sinc mechanism underlies these problems, but the corresponding operators encode different geometric information. One of the main purposes of the present paper is to keep these structures distinct while developing a common operator-theoretic framework for their analysis.

The first object is the genuine Gram operator associated with an index-dependent perturbed exponential family. If the frequencies are
$$
\lambda_n=n+\delta_n,
$$
then the Gram kernel is
$$
\sinc_\pi(n-m+\delta_n-\delta_m),
$$
and therefore depends simultaneously on the perturbations attached to both indices. This two-index dependence is essential and distinguishes the genuine Gram problem from the auxiliary shifted-sinc model that will be used throughout the paper.

The second object is the sign-modulated shifted-sinc synthesis operator
$$
(N_\delta a)_m
=
\sum_{n\in\Z}
a_n\sinc_\pi(n-m+s_n\delta),
\qquad
s_n\in\{-1,+1\},
$$
whose shift depends only on the input index. Equivalently,
$$
N_\delta=T_\delta P_+ + T_{-\delta}P_-,
$$
where $P_+$ and $P_-$ are the coordinate projections determined by the sign pattern. This auxiliary model is not itself the Gram operator of the perturbed family, but it admits an exact block structure that isolates the role of the sign pattern and makes possible a detailed spectral analysis.

A third structure appears for exponential systems generated by unions of shifted lattices on disconnected domains. After folding modulo one, the synthesis problem is represented by finite-dimensional matrix fibers. For unions of intervals of unequal lengths the fiber dimension may vary with the base point, but only finitely many different fibers occur. Thus the infinite-dimensional problem is again governed by explicit finite-dimensional data, now through a geometric rather than a periodic reduction.

These three settings are linked by the same shifted-sinc and Fourier-fiber mechanisms, but they answer different questions. In the perturbative problem the main issue is how the arrangement of local frequency shifts affects stability. In the disconnected-domain problem the relevant information is encoded by the integer fibers of the spatial set and by the corresponding phase matrices. The results below show that periodicity or finite-cell geometry can collapse either problem to exact finite-dimensional calculations, whereas genuinely nonperiodic configurations retain an essentially infinite-dimensional component.

\subsection{Main results and organization}

The first part of the paper concerns perturbed exponential systems and local frequency shifts. Section~\ref{sec:gram} introduces the genuine Gram framework and records the distinction between the actual Gram operator and the auxiliary models used later. Section~\ref{sec:shifted-group} develops the shifted-sinc operators
$$
T_t,
$$
showing that they form a unitary group on $\ell^2(\Z)$. This group structure provides the common operator language used throughout the perturbative part of the paper.

Section~\ref{sec:sign-model} then studies the auxiliary binary operator
$$
N_\delta=T_\delta P_+ + T_{-\delta}P_-.
$$
Besides estimates that are uniform over all sign patterns, the fixed-pattern problem admits an exact description. Section~\ref{sec:exact-pattern} proves that, if
$$
\rho_{\delta,s}
=
\|P_-T_{2\delta}P_+\|,
$$
then
$$
\inf\sigma(N_\delta^*N_\delta)
=
1-\rho_{\delta,s},
\qquad
\sup\sigma(N_\delta^*N_\delta)
=
1+\rho_{\delta,s}.
$$
Thus a single pattern-dependent compression norm determines both optimal squared synthesis bounds. In particular, $\rho_{\delta,s}<1$ is exactly the lower-stability condition for the synthesis operator onto its closed range.

Periodicity produces a further and stronger reduction. For a $p$-periodic binary pattern, Section~\ref{sec:periodic-binary} proves a factorization
$$
\mathcal N_{p,\delta,s}(\xi)
=
H_{p,\delta,s}E_{\delta,s}(\xi)V^*,
$$
in which the last two factors are unitary. Consequently, all singular values are independent of the fiber variable $\xi$, and the optimal bounds are the extreme eigenvalues of the single finite Gram matrix
$$
H_{p,\delta,s}^*H_{p,\delta,s}.
$$
The basis property is equivalent to a finite Vandermonde noncollision condition. This gives, in particular, the exact alternating bounds
$$
1\pm\sin(\pi\delta)
$$
and an explicit period-three example with optimal squared bounds
$$
1\pm\frac{\sqrt5}{3}
$$
at $\delta=1/4$.

The relation between periodic and nonperiodic binary patterns is developed further in Section~\ref{subsec:periodic-extremizers}. If $A_{\rm bin}(\delta)$ denotes the worst lower squared synthesis bound over all binary patterns and $A_p(\delta)$ the corresponding minimum over $p$-periodic patterns, we prove
$$
A_{\rm bin}(\delta)
=
\inf_{p\geq1}A_p(\delta).
$$
Thus the global extremal problem can be approximated by finite-dimensional periodic problems, even when no fixed period realizes the worst behavior. At the critical value $\delta=1/4$, a step pattern leads to a half-shifted Hilbert matrix and has vanishing lower bound, whereas every fixed periodic binary pattern remains invertible. Hence the critical loss of stability is invisible at every fixed period and emerges only asymptotically as the period tends to infinity. This leads naturally to finite-dimensional extremal questions and to a conjectural logarithmic critical degeneration, motivated by the asymptotics of finite Hilbert matrices.

We then return to the genuine Gram problem. Sections~\ref{sec:sesquilinear} and~\ref{sec:near-critical} separate the two mechanisms that enter near the critical regime. A four-block sesquilinear model first isolates the interaction of two sign partitions. The genuine Gram expansion then distinguishes same-sign drift from opposite-sign half-shift interaction and yields support-dependent first-order estimates. Section~\ref{sec:structured-neighborhood} gives a complementary support-independent perturbation theorem around any already stable structured exponential basis. The two results address different aspects of stability: the former retains detailed first-order information about the perturbation, whereas the latter provides a global neighborhood of a fixed stable system.

Section~\ref{sec:finite-alphabet} extends the local-shift theory from the binary alphabet $\{+\delta,-\delta\}$ to finitely many shifts $t_1,\ldots,t_L$. If $P_j$ denotes the coordinate projection associated with the shift $t_j$, the Gram blocks are
$$
G_{jk}=P_jT_{t_k-t_j}P_k.
$$
A matrix of block norms yields a general sufficient lower-stability criterion. The binary case is exceptional because its self-adjoint off-diagonal structure reduces the two spectral endpoints to a single scalar norm; this symmetry does not persist for three or more local shifts. Under periodicity, however, the finite-alphabet problem again reduces exactly to a single constant Vandermonde-type matrix.

The second part of the paper concerns disconnected spatial domains. We begin with two full shifted lattices on a union of two unit intervals. Sections~\ref{sec:integer-twointerval} and~\ref{sec:noninteger-twointerval} treat respectively integer and noninteger translations and reduce the synthesis operator to explicit $2\times2$ matrix fibers. This gives optimal bounds and complete basis criteria in both regimes.

These calculations are instances of a more general folding principle. Section~\ref{sec:finite-unions} first extends the construction to finitely many translated unit intervals and then to arbitrary finite unions of bounded intervals. For
$$
E=\bigcup_{j=1}^J[a_j,b_j),
\qquad
\Lambda=\bigcup_{k=1}^K(\Z+\delta_k),
$$
we associate to each $x\in[0,1)$ the integer fiber
$$
\Omega_x
=
\{m\in\Z:x+m\in E\}.
$$
Only finitely many different sets $\Omega_x$ occur on positive-measure cells. On a cell with fiber $\Omega_r$, the synthesis operator is represented by the finite phase matrix
$$
M_r
=
\bigl(\e^{2\pi i\delta_km}\bigr)_{
m\in\Omega_r,\,
1\leq k\leq K}.
$$
The extreme singular values of these matrices give the exact quantitative bounds, while their column rank, row rank, and square invertibility characterize respectively Riesz sequences, frames, and Riesz bases. This formulation also makes clear why total measure alone is insufficient: the local fiber multiplicity must be compatible with the number of frequency channels.

The finite-cell formulation has two further applications. In Section~\ref{subsec:fuglede-fibers}, periodicity of spectra for finite unions of intervals converts spectrality into simultaneous finite Hadamard conditions on the integer fibers. Comparing adjacent fibers then produces arithmetic restrictions on possible spectra and, under suitable transition hypotheses, a rigidity mechanism related to the one-dimensional Fuglede problem. In Section~\ref{subsec:commensurate-multirate}, a finite collection of rationally commensurate sampling lattices is decomposed into cosets of a common lattice. The resulting multirate system is therefore governed by the same finite-cell phase matrices, giving exact existence criteria, optimal condition numbers, and an explicit reconstruction procedure.

The paper concludes by separating the regimes in which exact finite-dimensional reductions are available from the genuinely nonperiodic problems that remain infinite-dimensional. Among the latter are the geometric characterization of
$$
\rho_{\delta,s}
=
\|P_-T_{2\delta}P_+\|
$$
for arbitrary nonperiodic patterns, the critical asymptotics of the periodic extremal problem, a support-independent near-critical theory for the genuine Gram operator, and extensions of the finite-cell framework beyond rationally commensurate lattice systems.

\section{The perturbed exponential Gram framework}\label{sec:gram}

\subsection{The one-dimensional Gram operator}

Let $(\delta_n)_{n\in\Z}$ be a real sequence. We consider the
exponential family
\begin{equation}\label{eq:B}
\mathcal B(\delta)
=
\bigl\{\e^{2\pi i(n+\delta_n)x}\bigr\}_{n\in\Z}
\subset L^2\left(-\tfrac12,\tfrac12\right).
\end{equation}
No smallness assumption is imposed at this stage.
Whenever a uniform stability result requires
$$\sup_{n\in\Z}|\delta_n|<\frac14,$$
this hypothesis will be stated explicitly. For finitely supported coefficient sequences, all Gram identities below
are valid without any smallness assumption on $(\delta_n)$.

We use the normalized sinc function
$$
\sinc_\pi(u):=
\begin{cases}
\dfrac{\sin(\pi u)}{\pi u},&u\neq0,\\[1ex]
1,&u=0.
\end{cases}
$$

Recall that $\mathcal B$ is a Riesz sequence if there exist constants $A,B>0$ such that, for every finitely supported sequence $\mathbf a=(a_n)_{n\in\Z}$,
\begin{equation}\label{eq:riesz-seq}
A\sum_{n\in\Z}|a_n|^2
\leq
\norm{\sum_{n\in\Z}a_n\e^{2\pi i(n+\delta_n)x}}_{L^2(-1/2,1/2)}^2
\leq
B\sum_{n\in\Z}|a_n|^2.
\end{equation}

Expanding the $L^2$ norm gives
\begin{align}
\norm{\sum_{n\in\Z}a_n\e^{2\pi i(n+\delta_n)x}}_{L^2(-1/2,1/2)}^2
&=\sum_{n,m\in\Z}a_n\overline{a_m}
\int_{-1/2}^{1/2}\e^{2\pi i(n-m+\delta_n-\delta_m)x}\dd x\notag\\
&=\sum_{n,m\in\Z}a_n\overline{a_m}
\sinc_\pi(n-m+\delta_n-\delta_m).
\label{eq:gram-kernel}
\end{align}
Accordingly, the Gram operator $G:\ell^2_0(\Z)\to\ell^2(\Z)$ is formally given by
\begin{equation}\label{eq:gram-operator}
(G\mathbf a)_m:=\sum_{n\in\Z}
\sinc_\pi(n-m+\delta_n-\delta_m)a_n,
\qquad m\in\Z,
\end{equation}
and \eqref{eq:gram-kernel} reads
$$
\norm{\sum_{n\in\Z}a_n\e^{2\pi i(n+\delta_n)x}}_{L^2(-1/2,1/2)}^2
=\ip{G\mathbf a}{\mathbf a}_{\ell^2}.
$$

Since
$$
\sinc_\pi(k+u)=(-1)^k\frac{\sin(\pi u)}{\pi(k+u)},
\qquad k\in\Z,\quad u\in\R,
$$
it is convenient to remove the oscillatory factor $(-1)^{n-m}$ by the unitary modulation
$$
(M\mathbf a)_n:=(-1)^n a_n,
\qquad n\in\Z.
$$
Define the unitarily equivalent operator $\widetilde G:=MGM$. Then
$$
\ip{G\mathbf a}{\mathbf a}_{\ell^2}
=\ip{\widetilde G(M\mathbf a)}{M\mathbf a}_{\ell^2},
$$
and the kernel of $\widetilde G$ is
\begin{equation}\label{eq:modulated-kernel}
\widetilde K_{m,n}=
\begin{cases}
1,&n=m,\\[1ex]
\dfrac{\sin(\pi(\delta_n-\delta_m))}
{\pi(n-m+\delta_n-\delta_m)},&n\neq m.
\end{cases}
\end{equation}
Thus the original Gram form and the modulated form are equivalent up to the unitary map $M$.

\begin{remark}\label{rem:auxiliary-distinction}
If $\delta_n=s_n\delta$ with $s_n\in\{+1,-1\}$, the genuine Gram kernel becomes
\begin{equation}\label{eq:true-sign-gram}
\sinc_\pi\bigl(n-m+\delta(s_n-s_m)\bigr).
\end{equation}
By contrast, the auxiliary sign operator studied in Section~\ref{sec:sign-model} has kernel
\begin{equation}\label{eq:aux-sign-kernel}
\sinc_\pi(n-m+s_n\delta).
\end{equation}
The first expression depends on both $s_n$ and $s_m$, whereas the second depends only on $s_n$. The auxiliary model is motivated by the Gram problem but is not the Gram operator itself.
\end{remark}

\subsection{Multidimensional shifted sinc operators}

For $\boldsymbol\theta=(\theta_1,\dots,\theta_d)\in\R^d$ and finitely supported $\mathbf a=(a_{\mathbf n})_{\mathbf n\in\Z^d}$, define
\begin{equation}\label{eq:DHT-dim}
(T_{\boldsymbol\theta}\mathbf a)_{\mathbf m}
:=\sum_{\mathbf n\in\Z^d}a_{\mathbf n}
\prod_{j=1}^d\sinc_\pi(n_j-m_j+\theta_j).
\end{equation}

\begin{lemma}\label{lem:innerprod}
Let $Q^d=[-\tfrac12,\tfrac12]^d$. Let $\mathbf a,\mathbf b\in\ell^2_0(\Z^d)$ and $\mathbf s,\mathbf t\in\R^d$. Then
\begin{align}
&\ip{\sum_{\mathbf n\in\Z^d}a_{\mathbf n}\e^{2\pi i\langle\mathbf n+\mathbf s,x\rangle}}
{\sum_{\mathbf m\in\Z^d}b_{\mathbf m}\e^{2\pi i\langle\mathbf m+\mathbf t,x\rangle}}_{L^2(Q^d)}\notag\\
&\qquad=\sum_{\mathbf n,\mathbf m\in\Z^d}a_{\mathbf n}\overline{b_{\mathbf m}}
\prod_{j=1}^d\sinc_\pi(n_j-m_j+s_j-t_j)\notag\\
&\qquad=\ip{T_{\mathbf s-\mathbf t}\mathbf a}{\mathbf b}_{\ell^2(\Z^d)}.
\label{eq:innerprod}
\end{align}
\end{lemma}

\begin{proof}
The integral factorizes coordinatewise:
\begin{align*}
\int_{Q^d}\e^{2\pi i\langle\mathbf n-\mathbf m+\mathbf s-\mathbf t,x\rangle}\dd x
&=\prod_{j=1}^d\int_{-1/2}^{1/2}
\e^{2\pi i(n_j-m_j+s_j-t_j)x_j}\dd x_j\\
&=\prod_{j=1}^d\sinc_\pi(n_j-m_j+s_j-t_j).
\end{align*}
Substituting this identity into the scalar product yields \eqref{eq:innerprod}.
\end{proof}

\section{The shifted sinc group}\label{sec:shifted-group}

For $t\in\R$ and finitely supported $a=(a_n)_{n\in\Z}$, define
\begin{equation}\label{eq:shifted-sinc}
(T_ta)_m:=\sum_{n\in\Z}a_n\sinc_\pi(n-m+t),
\qquad m\in\Z.
\end{equation}
The one-parameter family of shifted-sinc operators considered here is closely related to the discrete Hilbert-transform group studied by \cite{DeCarliSamad2016}. More precisely, if $\mathcal D_t$ denotes the family introduced in that work and $M$ is the unitary modulation $(Ma)_n=(-1)^n a_n$, then
$$
\mathcal D_t=MT_{-t}M,
\qquad t\in\mathbb R.
$$
Thus the group law, strong continuity, and $\ell^2$-isometry properties are already known in an equivalent normalization.

We recall the Fourier-multiplier representation and the resulting group properties below, since this formulation is particularly convenient for the pattern-dependent compressions and finite matrix reductions developed in the sequel.

Let $\Fcal:\ell^2(\Z)\to L^2(-\tfrac12,\tfrac12)$ be the discrete Fourier transform
\begin{equation}\label{eq:DFT}
(\Fcal a)(\xi)=\sum_{m\in\Z}a_m\e^{-2\pi i m\xi},
\qquad -\frac12<\xi<\frac12.
\end{equation}
\begin{theorem}[Fourier multiplier representation]\label{thm:multiplier}
For every $t\in\R$, the operator $T_t$ extends uniquely to a unitary operator on $\ell^2(\Z)$ and
\begin{equation}\label{eq:multiplier}
\Fcal(T_ta)(\xi)=\e^{-2\pi i t\xi}\Fcal a(\xi)
\end{equation}
for almost every $\xi\in(-\tfrac12,\tfrac12)$.
\end{theorem}

\begin{proof}
Let
$$
h_t(\xi)=\e^{-2\pi i t\xi},
\qquad -\frac12<\xi<\frac12,
$$
and let $\mathcal M_{h_t}$ denote multiplication by $h_t$ on
$L^2(-\tfrac12,\tfrac12)$. Define
$$
\widetilde T_t
:=
\Fcal^{-1}\mathcal M_{h_t}\Fcal.
$$
Since $|h_t(\xi)|=1$ almost everywhere, $\widetilde T_t$ is unitary
on $\ell^2(\Z)$.

It remains to verify that $\widetilde T_t$ agrees with
\eqref{eq:shifted-sinc} on finitely supported sequences. Let
$a\in\ell^2_0(\Z)$. For every $m\in\Z$,
\begin{align*}
(\widetilde T_ta)_m
&=
\int_{-1/2}^{1/2}
\e^{-2\pi i t\xi}
(\Fcal a)(\xi)
\e^{2\pi i m\xi}\,\dd\xi\\
&=
\sum_{n\in\Z}a_n
\int_{-1/2}^{1/2}
\e^{2\pi i(m-n-t)\xi}\,\dd\xi\\
&=
\sum_{n\in\Z}a_n
\sinc_\pi(m-n-t)\\
&=
\sum_{n\in\Z}a_n
\sinc_\pi(n-m+t).
\end{align*}
The interchange between the finite sum and the integral is immediate.
Thus $\widetilde T_t=T_t$ on $\ell^2_0(\Z)$, and consequently
$$
\Fcal(T_ta)(\xi)
=
\e^{-2\pi i t\xi}\Fcal a(\xi).
$$
The uniqueness of the extension follows from the density of
$\ell^2_0(\Z)$ in $\ell^2(\Z)$.
\end{proof}

\begin{remark}[A shift check at $t=1$]\label{rem:t-one}
For $t=1$,
\begin{align}
(T_1a)_m&=\sum_{n\in\Z}a_n\sinc_\pi(n-m+1).
\end{align}
Since $n-m+1$ is an integer, the sinc term vanishes unless $n-m+1=0$, that is, unless $n=m-1$. Hence
\begin{equation}\label{eq:T1-shift}
(T_1a)_m=a_{m-1}.
\end{equation}
The Fourier transform of this shift is
$$\Fcal(T_1a)(\xi)=\sum_m a_{m-1}\e^{-2\pi i m\xi}=\sum_k a_k\e^{-2\pi i(k+1)\xi}=\e^{-2\pi i\xi}\Fcal a(\xi),
$$
in agreement with \eqref{eq:multiplier}.
\end{remark}

\begin{remark}[Convolution convention]\label{rem:convolution}
With the standard discrete convolution
$$
(a*k)_m=\sum_{n\in\Z}a_nk(m-n),
$$
one has
$$(T_ta)_m=\sum_n a_n\sinc_\pi(n-m+t)=\sum_n a_n\sinc_\pi(m-n-t)=\sum_n a_nk_{-t}(m-n),
$$
where
\begin{equation}\label{eq:convolution}
T_ta=a*k_{-t},
\qquad k_{-t}(j)=\sinc_\pi(j-t).
\end{equation}
\end{remark}

\begin{proposition}[Group law, adjoint, and unitarity]\label{prop:group-law}
For all $s,t\in\R$,
\begin{equation}\label{eq:group-law}
T_tT_s=T_{t+s},
\qquad T_t^*=T_{-t},
\qquad \norm{T_ta}_{\ell^2}=\norm{a}_{\ell^2}.
\end{equation}
\end{proposition}

\begin{proof}
By Parseval and Theorem~\ref{thm:multiplier},
$$\Fcal(T_tT_sa)(\xi)=\e^{-2\pi i t\xi}\Fcal(T_sa)(\xi)=\e^{-2\pi i(t+s)\xi}\Fcal a(\xi)=\Fcal(T_{t+s}a)(\xi),$$
which gives $T_tT_s=T_{t+s}$. Multiplication by $\e^{-2\pi i t\xi}$ has adjoint multiplication by its complex conjugate $\e^{2\pi i t\xi}$, which is the multiplier of $T_{-t}$, and therefore $T_t^*=T_{-t}$. Finally,
$$\norm{T_ta}_{\ell^2}^2=\norm{\Fcal(T_ta)}_{L^2(-1/2,1/2)}^2=\int_{-1/2}^{1/2}|\e^{-2\pi i t\xi}|^2|\Fcal a(\xi)|^2\dd\xi=\int_{-1/2}^{1/2}|\Fcal a(\xi)|^2\dd\xi=\norm{a}_{\ell^2}^2.$$
\end{proof}

\begin{lemma}[Distance from the identity]\label{lem:distance-identity}
For $0\leq|t|\leq1$,
\begin{equation}\label{eq:distance-identity}
\norm{T_t-I}_{\ell^2\to\ell^2}
=2\sin\left(\frac{\pi|t|}{2}\right).
\end{equation}
In particular,
\begin{equation}\label{eq:distance-2delta}
\norm{T_{2\delta}-I}_{\ell^2\to\ell^2}=2\sin(\pi\delta),
\qquad 0\leq\delta<\frac12.
\end{equation}
\end{lemma}

\begin{proof}
From \eqref{eq:multiplier},
\begin{equation}\label{eq:distance-sup}
\norm{T_t-I}_{\ell^2\to\ell^2}
=\mathop{\mathrm{ess\,sup}}_{|\xi|\leq1/2}|\e^{-2\pi i t\xi}-1|.
\end{equation}
Moreover,
$$|\e^{-2\pi i t\xi}-1|^2=4\sin^2(\pi t\xi),$$
so
\begin{equation}\label{eq:distance-pointwise}
|\e^{-2\pi i t\xi}-1|=2|\sin(\pi t\xi)|.
\end{equation}
If $|t|\leq1$ and $|\xi|\leq1/2$, then $0\leq\pi|t\xi|\leq\pi/2$, and the sine is increasing on $[0,\pi/2]$. The essential supremum is obtained as $|\xi|\to1/2$, giving
\eqref{eq:distance-identity}. Taking $t=2\delta$ gives \eqref{eq:distance-2delta}.
\end{proof}

\section{The auxiliary sign-modulated shifted sinc operator}\label{sec:sign-model}

Fix $0\leq\delta<\tfrac12$ and a sign pattern $s=(s_n)_{n\in\Z}$ with $s_n\in\{+1,-1\}$. For finitely supported $a=(a_n)_{n\in\Z}$ define
\begin{equation}\label{eq:Ndelta}
(N_\delta a)_m:=\sum_{n\in\Z}a_n\sinc_\pi(n-m+s_n\delta),
\qquad m\in\Z.
\end{equation}
Let
$$
S_+:=\{n\in\Z:s_n=1\},
\qquad
S_-:=\{n\in\Z:s_n=-1\},
$$
and let $P_+$ and $P_-$ be the corresponding coordinate projections.

\begin{lemma}[Block decomposition]\label{lem:block-sign}
For every $a\in\ell^2(\Z)$,
\begin{equation}\label{eq:block-sign}
N_\delta=T_\delta P_+ +T_{-\delta}P_-.
\end{equation}
\end{lemma}

\begin{proof}
If $n\in S_+$, then
$$
\sinc_\pi(n-m+s_n\delta)=\sinc_\pi(n-m+\delta),
$$
whereas, if $n\in S_-$, then
$$
\sinc_\pi(n-m+s_n\delta)=\sinc_\pi(n-m-\delta).
$$
Splitting the sum in \eqref{eq:Ndelta} according to $S_+$ and $S_-$ gives \eqref{eq:block-sign}.
\end{proof}

\begin{proposition}[Uniform boundedness]\label{prop:uniform-sign}
For every sign pattern and every $0\leq\delta<\tfrac12$, the operator $N_\delta$ extends boundedly to $\ell^2(\Z)$ and
\begin{equation}\label{eq:uniform-sign}
\norm{N_\delta}_{\ell^2\to\ell^2}\leq\sqrt2.
\end{equation}
\end{proposition}

\begin{proof}
Using Lemma~\ref{lem:block-sign} and the unitarity of $T_{\pm\delta}$,
$$\norm{N_\delta a}=\norm{T_\delta(P_+a)+T_{-\delta}(P_-a)}\leq\norm{P_+a}+\norm{P_-a}\leq\sqrt{2\bigl(\norm{P_+a}^2+\norm{P_-a}^2\bigr)}
=\sqrt2\norm{a}.$$
\end{proof}

\begin{proposition}[Adjoint and square]\label{prop:adjoint-sign}
For an arbitrary sign pattern,
\begin{equation}\label{eq:adjoint-sign}
N_\delta^*=P_+T_{-\delta}+P_-T_\delta,
\end{equation}
and
\begin{equation}\label{eq:square-sign}
N_\delta^*N_\delta
=I+P_+T_{-2\delta}P_-+P_-T_{2\delta}P_+.
\end{equation}
\end{proposition}

\begin{proof}
From \eqref{eq:block-sign},
$$N_\delta^*=(T_\delta P_+ +T_{-\delta}P_-)^*=P_+T_{-\delta}+P_-T_\delta,$$
which proves \eqref{eq:adjoint-sign}. Hence
\begin{align}
N_\delta^*N_\delta
&=(P_+T_{-\delta}+P_-T_\delta)
(T_\delta P_+ +T_{-\delta}P_-)\\
&=P_+T_{-\delta}T_\delta P_+
+P_+T_{-\delta}T_{-\delta}P_-\notag\\
&\quad+P_-T_\delta T_\delta P_+
+P_-T_\delta T_{-\delta}P_-\\
&=P_++P_-+P_+T_{-2\delta}P_-+P_-T_{2\delta}P_+\\
&=I+P_+T_{-2\delta}P_-+P_-T_{2\delta}P_+,
\label{eq:square-sign-calculation}
\end{align}
as claimed.
\end{proof}

\begin{proposition}[Preliminary quadratic estimate]\label{prop:preliminary-quadratic}
For every $a\in\ell^2(\Z)$,
\begin{equation}\label{eq:preliminary-quadratic}
\bigl(1-2\sin(\pi\delta)\bigr)\norm{a}_2^2
\leq\norm{N_\delta a}_2^2
\leq\bigl(1+2\sin(\pi\delta)\bigr)\norm{a}_2^2.
\end{equation}
Consequently, this estimate gives a positive lower bound whenever
\begin{equation}\label{eq:one-sixth}
2\sin(\pi\delta)<1,
\qquad\text{equivalently}\qquad
0\leq\delta<\frac16.
\end{equation}
\end{proposition}

\begin{proof}
Put
$$
a^+=P_+a,
\qquad
a^-=P_-a.
$$
The supports of $a^+$ and $a^-$ are disjoint, so
\begin{equation}\label{eq:orthogonal-components}
\ip{a^+}{a^-}=0,
\qquad
\norm{a}_2^2=\norm{a^+}_2^2+\norm{a^-}_2^2.
\end{equation}
By \eqref{eq:block-sign},
\begin{equation}\label{eq:Ndelta-components}
N_\delta a=T_\delta a^+ +T_{-\delta}a^-.
\end{equation}
Therefore
\begin{align}
\norm{N_\delta a}_2^2
&=\ip{T_\delta a^++T_{-\delta}a^-}
{T_\delta a^++T_{-\delta}a^-}\\
&=\norm{T_\delta a^+}_2^2
+\norm{T_{-\delta}a^-}_2^2
+2\operatorname{Re}\ip{T_\delta a^+}{T_{-\delta}a^-}.
\label{eq:Ndelta-square}
\end{align}
Since both shifted operators are unitary,
\begin{equation}\label{eq:unitary-diagonal}
\norm{T_\delta a^+}_2^2+\norm{T_{-\delta}a^-}_2^2
=\norm{a^+}_2^2+\norm{a^-}_2^2
=\norm{a}_2^2.
\end{equation}
Using $T_{-\delta}^*=T_\delta$,
$$\ip{T_\delta a^+}{T_{-\delta}a^-}=\ip{T_{-\delta}^*T_\delta a^+}{a^-}=\ip{T_\delta T_\delta a^+}{a^-}=\ip{T_{2\delta}a^+}{a^-}.$$
Since $a^+\perp a^-$,
$$\ip{T_{2\delta}a^+}{a^-}=\ip{(T_{2\delta}-I)a^+}{a^-}+\ip{a^+}{a^-}=\ip{(T_{2\delta}-I)a^+}{a^-}.$$
Thus, by Cauchy--Schwarz and \eqref{eq:distance-2delta},
\begin{align}
\left|2\operatorname{Re}\ip{T_{2\delta}a^+}{a^-}\right|
&\leq2\left|\ip{(T_{2\delta}-I)a^+}{a^-}\right|\leq2\norm{(T_{2\delta}-I)a^+}_2\norm{a^-}_2\\
&\leq2\norm{T_{2\delta}-I}\norm{a^+}_2\norm{a^-}_2=4\sin(\pi\delta)\norm{a^+}_2\norm{a^-}_2.
\label{eq:cross-preliminary}
\end{align}
The elementary inequality $2xy\leq x^2+y^2$, applied to $x=\norm{a^+}_2$ and $y=\norm{a^-}_2$, gives
\begin{equation}\label{eq:two-xy}
2\norm{a^+}_2\norm{a^-}_2
\leq\norm{a^+}_2^2+\norm{a^-}_2^2
=\norm{a}_2^2.
\end{equation}
Consequently,
\begin{equation}\label{eq:cross-preliminary-final}
\left|2\operatorname{Re}\ip{T_{2\delta}a^+}{a^-}\right|
\leq2\sin(\pi\delta)\norm{a}_2^2.
\end{equation}
Combining \eqref{eq:Ndelta-square}, \eqref{eq:unitary-diagonal}, and \eqref{eq:cross-preliminary-final} gives \eqref{eq:preliminary-quadratic}.
\end{proof}

\begin{theorem}[Refined quadratic estimate]\label{thm:refined-quadratic}
For every sign pattern, every $a\in\ell^2(\Z)$, and every $0\leq\delta\leq\tfrac14$,
\begin{equation}\label{eq:refined-quadratic}
\bigl(1-\sin(2\pi\delta)\bigr)\norm{a}_2^2
\leq\norm{N_\delta a}_2^2
\leq\bigl(1+\sin(2\pi\delta)\bigr)\norm{a}_2^2.
\end{equation}
In particular, $N_\delta$ is bounded below for every
\begin{equation}\label{eq:quarter-range}
0\leq\delta<\frac14.
\end{equation}
\end{theorem}

\begin{proof}
Orthogonality permits the subtraction of any scalar multiple of the identity. More precisely, for every $c\in\C$,
\begin{equation}
\ip{T_{2\delta}a^+}{a^-}=\ip{(T_{2\delta}-cI)a^+}{a^-}+c\ip{a^+}{a^-}=\ip{(T_{2\delta}-cI)a^+}{a^-}.
\label{eq:subtract-scalar}    
\end{equation}
Assume
\begin{equation}\label{eq:delta-quarter-assumption}
0\leq\delta\leq\frac14,
\end{equation}
and set $\varphi=2\pi\delta$, $c=\cos\varphi$. By \eqref{eq:multiplier}, the multiplier of $T_{2\delta}$ is $\e^{-4\pi i\delta\xi}$. Put $\theta=4\pi\delta\xi$. Since $|\xi|\leq1/2$, $|\theta|\leq2\pi\delta=\varphi$.
Moreover, \eqref{eq:delta-quarter-assumption} implies $0\leq\varphi\leq\pi/2$, so $\cos\varphi\geq0$. We compute
$$\norm{T_{2\delta}-\cos\varphi\,I}=\sup_{|\theta|\leq\varphi}|\e^{-i\theta}-\cos\varphi|.$$
The square of the quantity inside the supremum is
\begin{align}
|\e^{-i\theta}-\cos\varphi|^2
&=(\e^{-i\theta}-\cos\varphi)(\e^{i\theta}-\cos\varphi)=1+\cos^2\varphi-\cos\varphi(\e^{i\theta}+\e^{-i\theta})\\
&=1+\cos^2\varphi-2\cos\varphi\cos\theta.
\label{eq:arc-distance-square}
\end{align}
For $|\theta|\leq\varphi\leq\pi/2$, the cosine is even and decreasing on $[0,\pi]$, hence
\begin{equation}\label{eq:cos-order}
\cos\theta\geq\cos\varphi.
\end{equation}
Because $\cos\varphi\geq0$, equations \eqref{eq:arc-distance-square}--\eqref{eq:cos-order} yield
$$|\e^{-i\theta}-\cos\varphi|^2\leq1+\cos^2\varphi-2\cos^2\varphi=1-\cos^2\varphi=\sin^2\varphi.$$
Equality occurs at $\theta=\pm\varphi$, so
\begin{equation}\label{eq:best-scalar-norm}
\norm{T_{2\delta}-\cos(2\pi\delta)I}
=\sin(2\pi\delta),
\qquad 0\leq\delta\leq\frac14.
\end{equation}
Using \eqref{eq:subtract-scalar} with this choice of $c$,
\begin{align}
\left|2\operatorname{Re}\ip{T_{2\delta}a^+}{a^-}\right|
&\leq2\left|\ip{(T_{2\delta}-\cos(2\pi\delta)I)a^+}{a^-}\right|\leq2\sin(2\pi\delta)\norm{a^+}_2\norm{a^-}_2\\
&\leq\sin(2\pi\delta)
\bigl(\norm{a^+}_2^2+\norm{a^-}_2^2\bigr)=\sin(2\pi\delta)\norm{a}_2^2.
\label{eq:refined-cross}
\end{align}
Substitution into \eqref{eq:Ndelta-square} gives \eqref{eq:refined-quadratic}.

For $0<\delta<1/4$,
\begin{equation}\label{eq:strict-improvement}
\sin(2\pi\delta)=2\sin(\pi\delta)\cos(\pi\delta)
<2\sin(\pi\delta),
\end{equation}
so \eqref{eq:refined-quadratic} is strictly stronger than \eqref{eq:preliminary-quadratic} in that range. The lower coefficient is positive precisely for \eqref{eq:quarter-range}.
\end{proof}

\begin{remark}
At $\delta=1/4$, the lower coefficient in \eqref{eq:refined-quadratic} is zero and the upper coefficient is $2$, consistently with the uniform estimate $\norm{N_\delta}\leq\sqrt2$. The theorem concerns the auxiliary kernel \eqref{eq:aux-sign-kernel}; it does not by itself prove the corresponding statement for the genuine Gram kernel \eqref{eq:true-sign-gram}.
\end{remark}

\begin{proposition}[Vectorwise two-sided estimate]\label{prop:vectorwise-sign}
For every $a\in\ell^2(\Z)$,
\begin{equation}\label{eq:vectorwise-sign}
\left|\norm{P_+a}-\norm{P_-a}\right|
\leq\norm{N_\delta a}
\leq\norm{P_+a}+\norm{P_-a}.
\end{equation}
\end{proposition}

\begin{proof}
The upper bound is the triangle inequality applied to \eqref{eq:block-sign}. The lower bound follows from the reverse triangle inequality and the unitarity of $T_{\pm\delta}$.
\end{proof}

\section{Exact pattern-dependent bounds for the binary shifted-sinc model}
\label{sec:exact-pattern}

We now pass from pattern-uniform estimates to the exact geometry of a fixed sign pattern. Retain the notation of Section~\ref{sec:sign-model}:
$$
N_\delta=T_\delta P_+ +T_{-\delta}P_-,
$$
where $P_+$ and $P_-$ are the coordinate projections associated with
$$
S_+=\{n:s_n=1\},
\qquad
S_-=\{n:s_n=-1\}.
$$
Proposition~\ref{prop:adjoint-sign} gives
$$
N_\delta^*N_\delta
=
I+P_+T_{-2\delta}P_-+P_-T_{2\delta}P_+.
$$
The only nontrivial interaction is therefore the compression of the unitary shift $T_{2\delta}$ between the two sign subspaces.

Define
\begin{equation}\label{eq:v3-C}
C_{\delta,s}
=
P_-T_{2\delta}P_+:
P_+\ell^2(\Z)\longrightarrow P_-\ell^2(\Z),
\qquad
\rho_{\delta,s}=\norm{C_{\delta,s}}.
\end{equation}

\begin{theorem}[Exact block spectrum]\label{thm:v3-block-spectrum}
Relative to the orthogonal decomposition
$$
\ell^2(\Z)=P_+\ell^2(\Z)\oplus P_-\ell^2(\Z),
$$
one has
\begin{equation}\label{eq:v3-block-gram}
N_\delta^*N_\delta
=
\begin{pmatrix}
I&C_{\delta,s}^*\\
C_{\delta,s}&I
\end{pmatrix}.
\end{equation}
Moreover,
\begin{equation}\label{eq:v3-spectrum-endpoints}
\inf\sigma(N_\delta^*N_\delta)=1-\rho_{\delta,s},
\qquad
\sup\sigma(N_\delta^*N_\delta)=1+\rho_{\delta,s}.
\end{equation}
Consequently,
\begin{equation}\label{eq:v3-exact-pattern-bounds}
(1-\rho_{\delta,s})\norm{a}_2^2
\le
\norm{N_\delta a}_2^2
\le
(1+\rho_{\delta,s})\norm{a}_2^2,
\qquad a\in\ell^2(\Z),
\end{equation}
and both constants are optimal.
\end{theorem}

\begin{proof}
The block identity follows directly from Proposition~\ref{prop:adjoint-sign}. Put
$$
K_{\delta,s}
=
\begin{pmatrix}
0&C_{\delta,s}^*\\
C_{\delta,s}&0
\end{pmatrix},
\qquad
N_\delta^*N_\delta=I+K_{\delta,s}.
$$
The operator $K_{\delta,s}$ is self-adjoint. Introduce the unitary involution
$$
J
=
\begin{pmatrix}
I&0\\
0&-I
\end{pmatrix}.
$$
Then
$$
JK_{\delta,s}J=-K_{\delta,s}.
$$
Unitary conjugation preserves the spectrum, hence
$$
\sigma(K_{\delta,s})
=
\sigma(-K_{\delta,s})
=
-\sigma(K_{\delta,s}).
$$
Thus the spectrum is symmetric about the origin.

Furthermore,
$$
K_{\delta,s}^2
=
\begin{pmatrix}
C_{\delta,s}^*C_{\delta,s}&0\\
0&C_{\delta,s}C_{\delta,s}^*
\end{pmatrix},
$$
so
$$
\norm{K_{\delta,s}}^2
=
\max\{\norm{C_{\delta,s}^*C_{\delta,s}},\norm{C_{\delta,s}C_{\delta,s}^*}\}
=
\norm{C_{\delta,s}}^2.
$$
Hence
$$
\norm{K_{\delta,s}}=\rho_{\delta,s}.
$$
Because a bounded self-adjoint operator has compact real spectrum and its norm is the maximum of the absolute values of the spectral points, at least one of $\pm\rho_{\delta,s}$ belongs to the spectrum. Spectral symmetry then forces both endpoints to belong to the spectrum. Therefore
$$
\inf\sigma(K_{\delta,s})=-\rho_{\delta,s},
\qquad
\sup\sigma(K_{\delta,s})=\rho_{\delta,s},
$$
and shifting by the identity yields \eqref{eq:v3-spectrum-endpoints}. The quadratic inequalities and their optimality follow from the spectral theorem.
\end{proof}

\begin{remark}[Approximate singular vectors and optimality]\label{rem:v3-approx-singular}
The proof above does not require $C_{\delta,s}$ to attain its norm. Nevertheless, optimality can also be seen concretely by approximate singular vectors. Choose unit vectors $x_k\in P_+\ell^2(\Z)$ such that
$$
\norm{C_{\delta,s}x_k}\longrightarrow\rho_{\delta,s}.
$$
Whenever $C_{\delta,s}x_k\neq0$, set
$$
y_k=\frac{C_{\delta,s}x_k}{\norm{C_{\delta,s}x_k}},
\qquad
z_k^\pm=\frac1{\sqrt2}(x_k,\pm y_k).
$$
Then $\norm{z_k^\pm}=1$ and
$$
\ip{K_{\delta,s}z_k^+}{z_k^+}\longrightarrow\rho_{\delta,s},
\qquad
\ip{K_{\delta,s}z_k^-}{z_k^-}\longrightarrow-\rho_{\delta,s}.
$$
Thus the Rayleigh quotients of $N_\delta^*N_\delta$ approach $1\pm\rho_{\delta,s}$ even when no norm-attaining singular vector exists.
\end{remark}

\begin{corollary}[Pattern-dependent Riesz criterion]\label{cor:v3-pattern-criterion}
The exponential family
$$
\left\{\e^{2\pi i(n+s_n\delta)x}\right\}_{n\in\Z}
$$
is a Riesz sequence in $L^2(-1/2,1/2)$ if and only if
\begin{equation}\label{eq:v3-rho-condition}
\rho_{\delta,s}<1.
\end{equation}
When \eqref{eq:v3-rho-condition} holds, the condition number of the synthesis operator onto its closed range is
\begin{equation}\label{eq:v3-condition-number}
\kappa
=
\sqrt{\frac{1+\rho_{\delta,s}}{1-\rho_{\delta,s}}}.
\end{equation}
The condition $\rho_{\delta,s}<1$ alone does not imply completeness. If the synthesis operator is also onto, then the family is a Riesz basis.
\end{corollary}

\begin{proof}
The Fourier-coordinate representation of synthesis is $N_\delta$. By Theorem~\ref{thm:v3-block-spectrum}, a positive lower synthesis bound is equivalent to $\rho_{\delta,s}<1$. A bounded-below operator is injective with closed range, which is precisely the synthesis characterization of a Riesz sequence. Surjectivity is the additional condition needed for a Riesz basis. The condition-number formula follows from the optimal singular values $\sqrt{1-\rho_{\delta,s}}$ and $\sqrt{1+\rho_{\delta,s}}$.
\end{proof}

The pattern-uniform bound of Theorem~\ref{thm:refined-quadratic} is recovered from the exact invariant. Since $P_-P_+=0$,
$$
C_{\delta,s}=P_-(T_{2\delta}-cI)P_+
$$
for every scalar $c$. With $c=\cos(2\pi\delta)$ and the scalar-approximation identity used in the proof of Theorem~\ref{thm:refined-quadratic},
\begin{equation}\label{eq:v3-rho-uniform}
\rho_{\delta,s}\le\sin(2\pi\delta),
\qquad 0\le\delta\le\frac14.
\end{equation}
Thus $\sin(2\pi\delta)$ is a pattern-uniform upper estimate, whereas $\rho_{\delta,s}$ is the exact invariant for the fixed partition $S_+\cup S_-$. 

\section{Periodic binary patterns and constant finite matrices}
\label{sec:periodic-binary}

Suppose now that $s_{n+p}=s_n$ for a period $p\ge1$. The periodic structure permits a complete finite-dimensional reduction. The key point is stronger than the existence of a matrix-valued symbol: the fiber variable can be factored entirely into a unitary input matrix.

Let
$$
J_p=\left[-\frac12,-\frac12+\frac1p\right),
\qquad
v_r=\bigl(\e^{-2\pi ijr/p}\bigr)_{j=0}^{p-1},
\qquad
Q_r=\frac1p v_rv_r^*,
$$
for $0\le r\le p-1$. The vectors $p^{-1/2}v_r$ form an orthonormal basis of $\C^p$. Put
\begin{equation}\label{eq:v3-V}
V=\frac1{\sqrt p}(v_0\ \cdots\ v_{p-1}),
\end{equation}
so that $V$ is unitary. For $t\in\R$ and $\xi\in J_p$, define
\begin{equation}\label{eq:v3-D}
D_t(\xi)
=
\operatorname{diag}\left(\e^{-2\pi it(\xi+j/p)}\right)_{j=0}^{p-1}.
\end{equation}
The standard period-$p$ fiberization sends the residue projection modulo $p$ to $Q_r$ and $T_t$ to multiplication by $D_t(\xi)$. Hence the binary synthesis symbol is
\begin{equation}\label{eq:v3-periodic-symbol}
\mathcal N_{p,\delta,s}(\xi)
=
\sum_{r=0}^{p-1}D_{s_r\delta}(\xi)Q_r.
\end{equation}

Define the constant matrix
\begin{equation}\label{eq:v3-Hbinary}
H_{p,\delta,s}
=
\left(
\frac1{\sqrt p}
\e^{-2\pi ij(r+s_r\delta)/p}
\right)_{j,r=0}^{p-1}
\end{equation}
and the diagonal unitary matrix
\begin{equation}\label{eq:v3-Ebinary}
E_{\delta,s}(\xi)
=
\operatorname{diag}\left(\e^{-2\pi i s_r\delta\xi}\right)_{r=0}^{p-1}.
\end{equation}

\begin{theorem}[Constant-matrix reduction for periodic binary patterns]\label{thm:v3-constant-binary}
For every $\xi\in J_p$,
\begin{equation}\label{eq:v3-constant-factorization}
\mathcal N_{p,\delta,s}(\xi)
=
H_{p,\delta,s}E_{\delta,s}(\xi)V^*.
\end{equation}
Consequently, all singular values of $\mathcal N_{p,\delta,s}(\xi)$ are independent of $\xi$. The optimal squared synthesis bounds are
\begin{equation}\label{eq:v3-periodic-bounds}
A_{p,\delta,s}
=
\lambda_{\min}(H_{p,\delta,s}^*H_{p,\delta,s}),
\qquad
B_{p,\delta,s}
=
\lambda_{\max}(H_{p,\delta,s}^*H_{p,\delta,s}).
\end{equation}
The associated exponential family is a Riesz basis if and only if $H_{p,\delta,s}$ is invertible.
\end{theorem}

\begin{proof}
Because $Q_qv_r=\delta_{qr}v_r$, the $r$-th column of $\mathcal N_{p,\delta,s}(\xi)V$ is
$$
D_{s_r\delta}(\xi)\frac{v_r}{\sqrt p}.
$$
Its $j$-th entry is
$$
\frac1{\sqrt p}
\e^{-2\pi i s_r\delta(\xi+j/p)}
\e^{-2\pi ijr/p}
=
\frac1{\sqrt p}
\e^{-2\pi ij(r+s_r\delta)/p}
\e^{-2\pi i s_r\delta\xi},
$$
which is exactly the $j$-th entry of the $r$-th column of $H_{p,\delta,s}E_{\delta,s}(\xi)$. Therefore
$$
\mathcal N_{p,\delta,s}(\xi)V
=
H_{p,\delta,s}E_{\delta,s}(\xi),
$$
and \eqref{eq:v3-constant-factorization} follows.

Both $E_{\delta,s}(\xi)$ and $V$ are unitary. Hence
$$
\mathcal N_{p,\delta,s}(\xi)^*
\mathcal N_{p,\delta,s}(\xi)
=
V E_{\delta,s}(\xi)^*
H_{p,\delta,s}^*H_{p,\delta,s}
E_{\delta,s}(\xi)V^*,
$$
so every fiber Gram matrix is unitarily conjugate to the same constant Gram matrix. This proves \eqref{eq:v3-periodic-bounds}.

To see sharpness directly, let $u\in\C^p$ be a unit eigenvector of $H_{p,\delta,s}^*H_{p,\delta,s}$ with eigenvalue $\mu$, and set
$$
F_u(\xi)=|J_p|^{-1/2}V E_{\delta,s}(\xi)^*u.
$$
Then $\int_{J_p}\norm{F_u(\xi)}^2\dd\xi=1$ and
$$
\int_{J_p}
\norm{\mathcal N_{p,\delta,s}(\xi)F_u(\xi)}^2\dd\xi
=
\norm{H_{p,\delta,s}u}^2
=
\mu.
$$
The fiberization is onto, so the smallest and largest eigenvalues are attained as optimal global constants.

If $H_{p,\delta,s}$ is invertible, then
$$
\mathcal N_{p,\delta,s}(\xi)^{-1}
=
V E_{\delta,s}(\xi)^*H_{p,\delta,s}^{-1},
$$
with norm independent of $\xi$, so the synthesis map is boundedly invertible. If $H_{p,\delta,s}$ is singular, choose $0\neq u\in\ker H_{p,\delta,s}$ and use the same fiber transformation to obtain a nonzero vector in the kernel of synthesis. Thus invertibility of $H_{p,\delta,s}$ is equivalent to the Riesz-basis property.
\end{proof}

\begin{remark}\label{rem:v3-right-unitary}
The factorization \eqref{eq:v3-constant-factorization} is a right-unitary reduction. It preserves singular values, rank, invertibility, kernel dimension, and the norm of the inverse. It does not assert that the individual matrices $\mathcal N_{p,\delta,s}(\xi)$ are unitarily similar to $H_{p,\delta,s}$.
\end{remark}

\begin{corollary}[Vandermonde collision criterion]\label{cor:v3-binary-vandermonde}
The periodic binary family is a Riesz basis if and only if
\begin{equation}\label{eq:v3-binary-collision}
r+s_r\delta\not\equiv q+s_q\delta\pmod p
\qquad(0\le q<r\le p-1).
\end{equation}
In particular, every fixed periodic binary pattern is a Riesz basis for
$$
0\le\delta<\frac12.
$$
\end{corollary}

\begin{proof}
Set
$$
z_r=\e^{-2\pi i(r+s_r\delta)/p}.
$$
Then $H_{p,\delta,s}$ is the normalized Vandermonde matrix
$$
(H_{p,\delta,s})_{jr}=p^{-1/2}z_r^j,
$$
so
$$
\det H_{p,\delta,s}
=
p^{-p/2}
\prod_{0\le q<r\le p-1}(z_r-z_q).
$$
Thus $H_{p,\delta,s}$ is invertible exactly when the nodes are pairwise distinct, which is equivalent to \eqref{eq:v3-binary-collision}.

If $0\le\delta<1/2$ and a collision occurred, then
$$
r-q+(s_r-s_q)\delta\in p\Z.
$$
Since $r-q$ is an integer, $(s_r-s_q)\delta\in\{0,\pm2\delta\}$ would also be an integer. Its absolute value is strictly less than one, hence it must vanish. Then $r-q\in p\Z$, impossible for distinct residues $q,r$ in $\{0,\ldots,p-1\}$.
\end{proof}

\begin{remark}\label{rem:v3-no-period-uniform}
The preceding corollary is a statement about each fixed periodic pattern. The lower bound may depend on both the period and the pattern; no positive constant uniform over all periods is asserted.
\end{remark}

\subsection{The alternating pattern}

Take
\begin{equation}\label{eq:v3-alternating-pattern}
s_n=(-1)^n.
\end{equation}
The frequency pairs are
$$
\lambda_{2k}=2k+\delta,
\qquad
\lambda_{2k+1}=2k+1-\delta,
$$
with within-pair separation
\begin{equation}\label{eq:v3-alt-gap}
\lambda_{2k+1}-\lambda_{2k}=1-2\delta.
\end{equation}
Thus the geometry already predicts a collision at $\delta=1/2$.

\begin{theorem}[Exact alternating bounds]\label{thm:v3-alternating}
For the pattern \eqref{eq:v3-alternating-pattern} and $0\le\delta\le1/2$,
\begin{equation}\label{eq:v3-alt-bounds}
\bigl(1-\sin(\pi\delta)\bigr)\norm{a}_2^2
\le
\norm{N_\delta a}_2^2
\le
\bigl(1+\sin(\pi\delta)\bigr)\norm{a}_2^2.
\end{equation}
Both constants are optimal. The associated exponential family is a Riesz basis for $0\le\delta<1/2$. At $\delta=1/2$ two neighboring frequencies coincide and the lower bound vanishes.
\end{theorem}

\begin{proof}
Use the period-two symbol with $s_0=1$ and $s_1=-1$. For $\xi\in[-1/2,0)$ let
$$
\alpha=2\pi\delta\xi,
\qquad
\beta=2\pi\delta(\xi+1/2)=\alpha+\pi\delta.
$$
In the fiber coordinates $(f(\xi),f(\xi+1/2))^T$,
$$
\mathcal N_{2,\delta}(\xi)
=
\begin{pmatrix}
\cos\alpha&-i\sin\alpha\\
-i\sin\beta&\cos\beta
\end{pmatrix}.
$$
The Gram matrix has trace $2$, while
$$
\det\mathcal N_{2,\delta}(\xi)
=
\cos\alpha\cos\beta+\sin\alpha\sin\beta
=
\cos(\alpha-\beta)
=
\cos(\pi\delta).
$$
Hence the Gram determinant is $\cos^2(\pi\delta)$ and its eigenvalues solve
$$
\lambda^2-2\lambda+\cos^2(\pi\delta)=0.
$$
For $0\le\delta\le1/2$ the eigenvalues are
$$
1-\sin(\pi\delta),
\qquad
1+\sin(\pi\delta),
$$
independently of $\xi$. Theorem~\ref{thm:v3-constant-binary} gives optimality and the basis property for $\delta<1/2$. At $\delta=1/2$, equation \eqref{eq:v3-alt-gap} gives
$$
2k+\frac12=2k+1-\frac12,
$$
so two exponentials become identical.
\end{proof}

\begin{corollary}\label{cor:v3-alt-rho}
For the alternating pattern,
$$
\rho_{\delta,s}=\sin(\pi\delta),
\qquad 0\le\delta\le\frac12.
$$
In particular, at $\delta=1/4$ the exact lower bound is
$$
1-\frac1{\sqrt2}>0.
$$
\end{corollary}

\begin{proof}
Compare Theorem~\ref{thm:v3-alternating} with Theorem~\ref{thm:v3-block-spectrum}.
\end{proof}

\subsection{A period-three binary configuration}

Consider the period-three pattern
\begin{equation}\label{eq:v3-period3-pattern}
s_{3k}=s_{3k+1}=1,
\qquad
s_{3k+2}=-1,
\end{equation}
and take $\delta=1/4$. The frequencies in one period are
$$
3k+\frac14,
\qquad
3k+\frac54,
\qquad
3k+\frac74.
$$

\begin{proposition}[Exact period-three bounds]\label{prop:v3-period3}
For \eqref{eq:v3-period3-pattern} at $\delta=1/4$, the associated exponential family is a Riesz basis with optimal squared bounds
\begin{equation}\label{eq:v3-period3-bounds}
A=1-\frac{\sqrt5}{3},
\qquad
B=1+\frac{\sqrt5}{3}.
\end{equation}
Moreover,
\begin{equation}\label{eq:v3-period3-rho}
\rho_{1/4,s}=\frac{\sqrt5}{3}.
\end{equation}
\end{proposition}

\begin{proof}
Subtract the common shift $1/4$ from the three residues. This multiplies the rows of the constant matrix by unimodular factors and therefore preserves singular values. With
$$
\omega=\e^{-2\pi i/3},
$$
the constant matrix is unitarily equivalent on the left to
$$
\widetilde H
=
\frac1{\sqrt3}
\begin{pmatrix}
1&1&1\\
1&\omega&-1\\
1&\omega^2&1
\end{pmatrix}.
$$
Using $1+\omega+\omega^2=0$ and $\overline\omega=\omega^2$,
$$
\widetilde H^*\widetilde H
=
\begin{pmatrix}
1&0&1/3\\
0&1&-2\omega^2/3\\
1/3&-2\omega/3&1
\end{pmatrix}
=
\begin{pmatrix}
I_2&c\\
c^*&1
\end{pmatrix},
\qquad
c=
\begin{pmatrix}
1/3\\
-2\omega^2/3
\end{pmatrix}.
$$
Since
$$
\norm{c}^2=\frac19+\frac49=\frac59,
$$
the one-dimensional subspace in the first two coordinates orthogonal to $c$ has eigenvalue $1$. On the orthogonal complement, the Gram matrix reduces to
$$
\begin{pmatrix}
1&\norm c\\
\norm c&1
\end{pmatrix},
$$
whose eigenvalues are $1\pm\sqrt5/3$. This proves \eqref{eq:v3-period3-bounds}. Comparing with the exact binary formula of Theorem~\ref{thm:v3-block-spectrum} gives \eqref{eq:v3-period3-rho}.
\end{proof}

\begin{remark}[General period-three formula]\label{rem:v3-period3-general}
For the same pattern and $0<\delta<1/2$, the two positive-sign columns are orthogonal and the negative column interacts with their two-dimensional span. If $\rho_3(\delta)$ denotes the norm of this interaction, then
$$
\sigma(H_{3,\delta,s}^*H_{3,\delta,s})
=
\{1,1-\rho_3(\delta),1+\rho_3(\delta)\},
$$
with
\begin{equation}\label{eq:v3-rho3-general}
\rho_3(\delta)
=
\sqrt{
1-
\left(
\frac{\sin(2\pi\delta)}{3\sin(2\pi\delta/3)}
\right)^2
}.
\end{equation}
At $\delta=1/4$ this equals $\sqrt5/3$, while at $\delta=1/2$ it equals $1$. The latter degeneration again corresponds to a collision:
$$
3k+1+\frac12=3k+2-\frac12.
$$
\end{remark}

\section{Critical endpoint and periodic extremizers}
\label{sec:critical-extremizers}

\subsection{Periodic approximation of worst-case binary stability}
\label{subsec:periodic-extremizers}

For a binary pattern $s=(s_n)_{n\in\Z}$ and $0\leq\delta<1/2$, write
$$
A(\delta,s)
:=\inf_{0\neq a\in\ell^2(\Z)}
\frac{\norm{N_{\delta,s}a}_2^2}{\norm a_2^2}
=1-\rho_{\delta,s},
$$
where the last identity follows from the exact block spectral formula of Section~\ref{sec:exact-pattern}. Define the worst binary lower bound
\begin{equation}\label{eq:worst-binary-bound}
A_{\rm bin}(\delta)
:=\inf_{s\in\{-1,+1\}^{\Z}}A(\delta,s).
\end{equation}
For $p\geq1$, define the periodic extremal quantity
\begin{equation}\label{eq:periodic-extremal-Ap}
A_p(\delta)
:=
\min_{s\in\{-1,+1\}^{p}}
\lambda_{\min}\!\left(H_{p,\delta,s}^*H_{p,\delta,s}\right),
\end{equation}
where $H_{p,\delta,s}$ is the constant matrix in Theorem~\ref{thm:v3-constant-binary}.

\begin{theorem}[Periodic approximation of the worst binary lower bound]
\label{thm:periodic-approx-worst}
For every $0\leq\delta<1/2$,
\begin{equation}\label{eq:periodic-approx-worst}
A_{\rm bin}(\delta)
=
\inf_{p\geq1} A_p(\delta).
\end{equation}
Equivalently,
$$
\sup_{s\in\{-1,+1\}^{\Z}}\rho_{\delta,s}
=
\sup_{p\geq1}\ \max_{s\in\{-1,+1\}^{p}}
\rho_{\delta,s}.
$$
\end{theorem}

\begin{proof}
Every periodic pattern is in particular a binary pattern on $\Z$, so
$$
A_{\rm bin}(\delta)\leq \inf_{p\geq1}A_p(\delta).
$$
For the reverse inequality fix an arbitrary binary pattern $s$ and $\varepsilon>0$. Since finitely supported sequences are dense in $\ell^2(\Z)$ and $N_{\delta,s}$ is bounded, there exists $0\neq a\in\ell^2_0(\Z)$ such that
$$
\frac{\norm{N_{\delta,s}a}_2^2}{\norm a_2^2}
<A(\delta,s)+\varepsilon.
$$
Choose an integer $p$ strictly larger than the diameter of $\operatorname{supp}a$. There is a $p$-periodic pattern $\widetilde s$ that agrees with $s$ on $\operatorname{supp}a$. Since the sign in the definition of $N_{\delta,s}$ is attached only to the input index,
$$
N_{\delta,\widetilde s}a=N_{\delta,s}a.
$$
Hence
$$
A_p(\delta)
\leq A(\delta,\widetilde s)
\leq
\frac{\norm{N_{\delta,\widetilde s}a}_2^2}{\norm a_2^2}
< A(\delta,s)+\varepsilon.
$$
Taking the infimum over $s$ and then letting $\varepsilon\downarrow0$ gives the reverse inequality.
\end{proof}

Theorem~\ref{thm:periodic-approx-worst} is useful because Theorem~\ref{thm:v3-constant-binary} converts every term on the right-hand side of \eqref{eq:periodic-approx-worst} into a finite-dimensional eigenvalue problem. Thus the global binary extremal problem is the limit of a sequence of explicit finite Vandermonde problems, even though no single finite period need realize the worst nonperiodic behavior.

\subsection{Critical degeneration and extremal conjectures}
\label{subsec:critical-degeneration}
The critical value $\delta=1/4$ makes this phenomenon particularly transparent. Let
\begin{equation}\label{eq:step-pattern}
s_n^{\rm step}
=
\begin{cases}
+1,&n\leq0,\\
-1,&n\geq1.
\end{cases}
\end{equation}
Let $P_+$ and $P_-$ denote the corresponding coordinate projections.

\begin{lemma}[The critical step pattern and a half-shifted Hilbert matrix]
\label{lem:step-hilbert}
For $\delta=1/4$, the cross compression
$$
C_{\rm step}=P_-T_{1/2}P_+
$$
is unitarily equivalent to the Hankel matrix
\begin{equation}\label{eq:half-hilbert-matrix}
\frac1\pi
\left(
\frac{1}{j+k+1/2}
\right)_{j,k\geq0}.
\end{equation}
Moreover,
\begin{equation}\label{eq:step-rho-one}
\norm{C_{\rm step}}=1,
\qquad
\rho_{1/4,s^{\rm step}}=1,
\qquad
A(1/4,s^{\rm step})=0.
\end{equation}
\end{lemma}

\begin{proof}
For $n=-j\leq0$ and $m=k+1\geq1$,
$$
\sinc_\pi(n-m+1/2)
=
\sinc_\pi\bigl(-(j+k+1/2)\bigr)
=
\frac{(-1)^{j+k}}{\pi(j+k+1/2)}.
$$
The factor $(-1)^{j+k}$ is removed by diagonal unitary modulations on the input and output half-lines, which gives \eqref{eq:half-hilbert-matrix}.

Let
$$
\mathcal H_{1/2}
=
\left(\frac{1}{j+k+1/2}\right)_{j,k\geq0},
\qquad
\mathcal H_{1}
=
\left(\frac{1}{j+k+1}\right)_{j,k\geq0}.
$$
Their difference is Hilbert--Schmidt because
$$
\sum_{j,k\geq0}
\left|
\frac{1}{j+k+1/2}-\frac{1}{j+k+1}
\right|^2<\infty.
$$
The classical Hilbert matrix $\mathcal H_1$ has essential norm $\pi$~\cite{Rosenblum1958HilbertII}; hence $\mathcal H_{1/2}$ has essential norm $\pi$ as well. Therefore \eqref{eq:half-hilbert-matrix} has norm at least one. On the other hand $C_{\rm step}$ is a compression of the unitary operator $T_{1/2}$, so $\norm{C_{\rm step}}\leq1$. Thus equality holds. The remaining statements follow from the exact formula $A(\delta,s)=1-\rho_{\delta,s}$.
\end{proof}

\begin{corollary}[Critical degeneration through increasing periods]
\label{cor:critical-periodic-degeneration}
At the critical value $\delta=1/4$,
\begin{equation}\label{eq:critical-inf-periodic-zero}
\inf_{p\geq1}
\min_{s\in\{-1,+1\}^{p}}
\lambda_{\min}\!\left(H_{p,1/4,s}^*H_{p,1/4,s}\right)
=0.
\end{equation}
Nevertheless, for every fixed $p$ and every fixed binary $p$-periodic pattern, the matrix $H_{p,1/4,s}$ is invertible.
\end{corollary}

\begin{proof}
The first statement follows from Theorem~\ref{thm:periodic-approx-worst} and Lemma~\ref{lem:step-hilbert}. The second follows from the collision criterion in Corollary~\ref{cor:v3-binary-vandermonde}, since $1/4<1/2$.
\end{proof}

Thus the loss of stability at the critical value $\delta=1/4$ is not detected by any fixed periodic pattern: every fixed-period problem remains invertible, whereas the infimum over increasing periods vanishes. The critical obstruction is therefore genuinely asymptotic in the period. This suggests an extremal finite-dimensional problem in its own right. This is consistent with the critical role of the threshold $\delta=1/4$ in Kadec's $1/4$ theorem; see \cite{AlemanyNitzan2025}. For a related but distinct discussion of the thresholds $1/4$ and $1/2$ in perturbed trigonometric interpolation and quadrature, see \cite{AustinTrefethen2017}.

For $p\geq2$, let $s^{\rm blk,p}$ be a balanced two-block pattern, for instance
\begin{equation}\label{eq:balanced-block-pattern}
s_r^{\rm blk,p}
=
\begin{cases}
+1,&0\leq r<\lceil p/2\rceil,\\
-1,&\lceil p/2\rceil\leq r<p.
\end{cases}
\end{equation}
Cyclic translations and a global sign change give equivalent candidates for the extremal problem.

\begin{conjecture}[Balanced two-block extremizer]
\label{conj:two-block-extremizer}
At $\delta=1/4$, for every $p\geq2$,
\begin{equation}\label{eq:two-block-extremizer}
A_p(1/4)
=
\lambda_{\min}\!\left(
H_{p,1/4,s^{\rm blk,p}}^*
H_{p,1/4,s^{\rm blk,p}}
\right).
\end{equation}
More generally, among binary patterns with a prescribed number of plus and minus signs, the minimum is attained by a pattern in which each sign occupies one contiguous cyclic block.
\end{conjecture}

Exhaustive numerical enumeration in double precision for $2\leq p\leq18$ supports Conjecture~\ref{conj:two-block-extremizer}. At $\delta=1/4$ the corresponding minima for $p=2,4,8,12,16,18$ are approximately
$$
0.292893,
\quad
0.206647,
\quad
0.150762,
\quad
0.127577,
\quad
0.114147,
\quad
0.109242,
$$
respectively. A reproducible computation is provided separately.

The appearance of the half-shifted Hilbert matrix in Lemma~\ref{lem:step-hilbert} suggests a logarithmically slow critical degeneration. For the classical $N\times N$ Hilbert matrix $\mathcal H_N=(1/(j+k+1))_{j,k=0}^{N-1}$, de Bruijn and Wilf~\cite{deBruijnWilf1962} proved
$$
\pi-\norm{\mathcal H_N}
=\frac{\pi^5}{2(\log N)^2}
+O\!\left(\frac{\log\log N}{(\log N)^3}\right).
$$
The periodic two-block matrices are not identical to ordinary finite sections of the Hilbert matrix; they contain a cyclic, periodized interaction. Therefore the Hilbert asymptotic cannot be transferred without a comparison theorem.

The currently accessible numerical range is too short to distinguish reliably between competing logarithmic decay rates. Accordingly, the following conjecture should be regarded as a heuristic suggested by the Hilbert-matrix asymptotics, rather than as an asymptotic law strongly supported by the present numerical data.
\begin{conjecture}[Critical logarithmic degeneration]
\label{conj:critical-logarithmic}
There exist constants $0<c<C<\infty$ such that
\begin{equation}\label{eq:critical-logarithmic}
\frac{c}{(\log p)^2}
\leq
A_p(1/4)
\leq
\frac{C}{(\log p)^2}
\end{equation}
for all sufficiently large $p$.
\end{conjecture}

A possible route to Conjecture~\ref{conj:critical-logarithmic} is to compare the cross block of the balanced periodic pattern with a finite section of the operator in \eqref{eq:half-hilbert-matrix}, controlling the wrap-around terms uniformly in $p$. Combined with Conjecture~\ref{conj:two-block-extremizer}, such a comparison would turn the sharp critical behavior of the full binary problem into an explicit finite-section asymptotic.

\section{From sign-pattern interactions to the genuine Gram form}
\label{sec:genuine-gram-analysis}
The preceding sections give a complete description of the auxiliary binary synthesis model. We now return to the genuine Gram problem. We first introduce a two-pattern sesquilinear form that isolates the interaction of shifted-sinc blocks, and then apply the same structure to the near-critical Gram kernel.
\subsection{A four-block sesquilinear model}\label{sec:sesquilinear}

Fix $0\leq\delta,\sigma<\tfrac12$.
Let $s=(s_n)_{n\in\Z}$ and $t=(t_m)_{m\in\Z}$, with
$s_n,t_m\in\{\pm1\}$.
For finitely supported sequences
$a,b\in\ell^2_0(\Z)$, define
\begin{equation}\label{eq:sesquilinear}
\mathcal T(a,b):=\sum_{n,m\in\Z}a_n\overline{b_m}
\sinc_\pi(n-m+s_n\delta-t_m\sigma).
\end{equation}
Write
$$
S_\pm:=\{n\in\Z:s_n=\pm1\},
\qquad
U_\pm:=\{m\in\Z:t_m=\pm1\},
$$
and let $P_\pm^s$ and $P_\pm^t$ be the corresponding coordinate projections.

\begin{lemma}[Four-block decomposition]\label{lem:four-block}
The form in \eqref{eq:sesquilinear} extends uniquely to a bounded
sesquilinear form on
$\ell^2(\Z)\times\ell^2(\Z)$. More precisely,
\begin{equation}\label{eq:four-block-bound}
|\mathcal T(a,b)|
\leq
2\|a\|_{\ell^2}\|b\|_{\ell^2}.
\end{equation}
For all $a,b\in\ell^2(\Z)$,
\begin{equation}\label{eq:four-block}
\mathcal T(a,b)=
\sum_{\varepsilon,\eta\in\{\pm1\}}
\left\langle
T_{\varepsilon\delta-\eta\sigma}
(P_\varepsilon^s a),
P_\eta^t b
\right\rangle_{\ell^2}.
\end{equation}
\end{lemma}

\begin{proof}
For finitely supported sequences, split the defining sum according to
the four sign combinations
$$
(s_n,t_m)=(\varepsilon,\eta).
$$
On each block,
$$
\sinc_\pi(n-m+s_n\delta-t_m\sigma)
=
\sinc_\pi(n-m+\varepsilon\delta-\eta\sigma),
$$
which gives \eqref{eq:four-block}.

By Cauchy--Schwarz and the unitarity of the shifted-sinc operators,
\begin{align*}
|\mathcal T(a,b)|
&\leq
\sum_{\varepsilon,\eta\in\{\pm1\}}
\|P_\varepsilon^s a\|_{\ell^2}
\|P_\eta^t b\|_{\ell^2}\\
&=
\bigl(\|P_+^s a\|_{\ell^2}+\|P_-^s a\|_{\ell^2}\bigr)
\bigl(\|P_+^t b\|_{\ell^2}+\|P_-^t b\|_{\ell^2}\bigr)\\
&\leq
2\|a\|_{\ell^2}\|b\|_{\ell^2}.
\end{align*}
The extension now follows uniquely by density.
\end{proof}

\begin{proposition}[Structured sign patterns]\label{prop:structured-patterns}
If both sign patterns are constant, then
$$
|\mathcal T(a,b)|\leq1
$$
for unit vectors. If exactly one sign pattern is constant, then
$$
|\mathcal T(a,b)|\leq\sqrt2.
$$
\end{proposition}
\begin{proof}
If both sign patterns are constant, write
$$
s_n\equiv\varepsilon,
\qquad
t_m\equiv\eta,
\qquad
\varepsilon,\eta\in\{\pm1\}.
$$
Only one block in \eqref{eq:four-block} survives, and
$$
\mathcal T(a,b)
=
\ip{T_{\varepsilon\delta-\eta\sigma}a}{b}.
$$
Since $T_{\varepsilon\delta-\eta\sigma}$ is unitary,
$$
|\mathcal T(a,b)|
\leq
\norm{a}_2\norm{b}_2
=1.
$$
In the special case $\varepsilon=\eta$ and $\delta=\sigma$, this
reduces to $T_0=I$ and hence
$\mathcal T(a,b)=\ip{a}{b}$.

If exactly one sign pattern is constant, only two blocks survive,
and the same argument used in
Lemma~\ref{lem:four-block} gives
$$
|\mathcal T(a,b)|\leq\sqrt2.
$$
\end{proof}

\begin{proposition}[Absence of a uniform positive lower bound]\label{prop:no-lower-sesquilinear}
There is no constant $A>0$ such that
$$
A\leq|\mathcal T(a,b)|
\qquad\text{for all }\norm{a}=\norm{b}=1.
$$
\end{proposition}

\begin{proof}
Define
$$
\mathcal S:=\sum_{\varepsilon,\eta\in\{\pm1\}}
P_\eta^tT_{\varepsilon\delta-\eta\sigma}P_\varepsilon^s,
$$
so that
$$
\mathcal T(a,b)=\ip{\mathcal Sa}{b}.
$$
If $\mathcal Sa\neq0$, choose a unit vector $b$ orthogonal to $\mathcal Sa$. If $\mathcal Sa=0$, every unit vector $b$ gives the same conclusion. In either case, $\mathcal T(a,b)=0$.
\end{proof}

\subsection{The near-critical regime of the genuine Gram form}\label{sec:near-critical}

Assume
\begin{equation}\label{eq:near-critical-delta}
\delta_n=s_n\left(\frac14-\varepsilon_n\right),
\qquad s_n\in\{\pm1\},
\qquad \varepsilon_n\geq0.
\end{equation}
Let $M$ be the unitary modulation introduced in Section~\ref{sec:gram}. It is enough to study
\begin{equation}\label{eq:qtilde}
\widetilde{\mathcal Q}(b):=\ip{\widetilde Gb}{b}
=\sum_{n,m\in\Z}b_n\overline{b_m}\widetilde K_{m,n},
\end{equation}
where $\widetilde K_{m,n}$ is given by \eqref{eq:modulated-kernel}. Since $\widetilde G=MGM$,
$$
\widetilde{\mathcal Q}(b)\geq0.
$$
Let
$$
S_\pm:=\{n\in\Z:s_n=\pm1\},
\qquad
b^\pm:=P_\pm b.
$$
For later use, for $\theta\in\R\setminus\Z$ define the modulated shifted sinc operator
\begin{equation}\label{eq:modulated-shifted-operator}
\mathcal H_\theta:=MT_\theta M,
\qquad
(\mathcal H_\theta a)_m
=\sum_{n\in\Z}a_n\frac{\sin(\pi\theta)}{\pi(n-m+\theta)}.
\end{equation}
It is unitary because it is unitarily equivalent to $T_\theta$. In the notation of \cite{DeCarliSamad2016}, this operator coincides with $\mathcal D_{-\theta}$. In particular, the half-shifted operator $H_{1/2}$ belongs to their unitary one-parameter group. The new feature of the present analysis is the interaction of these operators with the coordinate projections determined by the sign pattern, rather than the unitarity of the underlying discrete Hilbert transform.

\subsection{The critical form}

If $\varepsilon_n\equiv0$, that is, $\delta_n=s_n/4$, the same-sign terms vanish for $n\neq m$, whereas the opposite-sign terms produce a half-integer shift. Hence
\begin{equation}\label{eq:qcrit}
\widetilde{\mathcal Q}_{\mathrm{crit}}(b)
=\norm{b}_2^2+2\operatorname{Re}\ip{\mathcal H_{1/2}b^+}{b^-}.
\end{equation}
Therefore
\begin{equation}\label{eq:qcrit-bounds}
\norm{b}_2^2-2\norm{b^+}_2\norm{b^-}_2
\leq\widetilde{\mathcal Q}_{\mathrm{crit}}(b)
\leq\norm{b}_2^2+2\norm{b^+}_2\norm{b^-}_2
\leq2\norm{b}_2^2.
\end{equation}

\subsection{Kernel expansions}

Let $b\in\ell^2_0(\Z)$ and write
\begin{equation}\label{eq:support-parameters}
S:=\operatorname{supp}(b),
\qquad N:=|S|,
\qquad \eta_S:=\sup_{n\in S}\varepsilon_n,
\qquad \omega_S:=\sup_{n,m\in S}|\varepsilon_n-\varepsilon_m|.
\end{equation}
Assume, for convenience, that $\eta_S\leq1/8$.

\begin{lemma}[Same-sign and opposite-sign expansions]\label{lem:kernel-expansions}
Let $n,m\in S$, $n\neq m$. If $s_n=s_m$, then
\begin{equation}\label{eq:same-expansion}
\widetilde K_{m,n}
=s_n\frac{\varepsilon_m-\varepsilon_n}{n-m}
+r^{\mathrm{same}}_{m,n},
\end{equation}
where
\begin{equation}\label{eq:same-remainder}
|r^{\mathrm{same}}_{m,n}|
\leq C\frac{|\varepsilon_m-\varepsilon_n|^2}{|n-m|}.
\end{equation}

If $s_n\neq s_m$, then
\begin{equation}\label{eq:cross-expansion}
\widetilde K_{m,n}
=\frac{\sin\!\left(\pi(s_n-s_m)/4\right)}
{\pi\left(n-m+(s_n-s_m)/4\right)}
+r^{\mathrm{cross}}_{m,n},
\end{equation}
where
\begin{equation}\label{eq:cross-remainder}
|r^{\mathrm{cross}}_{m,n}|
\leq C\left(
\frac{\eta_S}{|n-m+(s_n-s_m)/4|^2}
+\frac{\eta_S^2}{|n-m+(s_n-s_m)/4|}
\right).
\end{equation}
\end{lemma}

\begin{proof}
Suppose first that $s_n=s_m=s$. Set $k=n-m$, $h=\varepsilon_m-\varepsilon_n$.
Then $\delta_n-\delta_m=s h$, and, by \eqref{eq:modulated-kernel},
\begin{equation}\label{eq:same-kernel-exact}
\widetilde K_{m,n}=\frac{\sin(\pi s h)}{\pi(k+s h)}.
\end{equation}
Since $\sin(\pi s h)=s\sin(\pi h)$,
$$\widetilde K_{m,n}-s\frac{h}{k}=s\left(\frac{\sin(\pi h)}{\pi(k+s h)}-\frac{h}{k}\right)=s\frac{k\sin(\pi h)-\pi h(k+s h)}{\pi k(k+s h)}=s\frac{k(\sin(\pi h)-\pi h)-\pi s h^2}{\pi k(k+s h)}.$$
Because $|h|\leq1/4$ and $k\in\Z\setminus\{0\}$, one has $|k+s h|\geq|k|/2$. Moreover, $|\sin(\pi h)-\pi h|\leq C|h|^3$. It follows that
$$\left|\widetilde K_{m,n}-s\frac{h}{k}\right|\leq C\left(\frac{|h|^3}{|k|}+\frac{|h|^2}{|k|^2}\right)\leq C\frac{|h|^2}{|k|},$$
which proves \eqref{eq:same-expansion}--\eqref{eq:same-remainder}.

Suppose now that $s_n\neq s_m$. Then $s_m=-s_n$. Set $s=s_n$, $k=n-m$, and $u=\varepsilon_n+\varepsilon_m$. We have
$\delta_n-\delta_m$ $=s\left(\frac12-u\right)$. Let
$x_0=k+\frac{s}{2}$ $=k+\frac{s_n-s_m}{4}$.
Then
\begin{equation}
\widetilde K_{m,n}=\frac{\sin(\pi(s/2-su))}{\pi(x_0-su)}=\frac{s\cos(\pi u)}{\pi(x_0-su)},
\label{eq:cross-kernel-exact}
\end{equation}
whereas 
$$\frac{\sin(\pi s/2)}{\pi x_0}=\frac{s}{\pi x_0}.$$
Comparing \eqref{eq:cross-kernel-exact} with this critical modulated kernel gives
$$\widetilde K_{m,n}-\frac{s}{\pi x_0}=\frac{s}{\pi}\left(\frac{\cos(\pi u)}{x_0-su}-\frac{1}{x_0}\right)=\frac{s}{\pi}\left(
\frac{\cos(\pi u)-1}{x_0-su}
+\frac{s u}{x_0(x_0-su)}
\right).$$
Since $|x_0|\geq1/2$, $u\leq2\eta_S\leq1/4$, and therefore $|x_0-su|\geq|x_0|/2$, while $|\cos(\pi u)-1|\leq C u^2$, we obtain
$$|r^{\mathrm{cross}}_{m,n}|\leq C\left(\frac{u}{|x_0|^2}+\frac{u^2}{|x_0|}\right)\leq C\left(
\frac{\eta_S}{|n-m+(s_n-s_m)/4|^2}
+\frac{\eta_S^2}{|n-m+(s_n-s_m)/4|}
\right).$$
This proves \eqref{eq:cross-expansion}--\eqref{eq:cross-remainder}.
\end{proof}

\begin{remark}\label{rem:modulated-cross-sign}
The critical shift is $(s_n-s_m)/4=\pm1/2$. Because the alternating factor has already been removed in \eqref{eq:modulated-kernel}, the relevant critical operator is $\mathcal H_{1/2}=MT_{1/2}M$, rather than $T_{1/2}$ itself.
\end{remark}

Define the same-sign drift term
\begin{equation}\label{eq:drift}
D_S(b):=\sum_{\substack{n,m\in S\\ n\neq m,\ s_n=s_m}}
b_n\overline{b_m}\,
 s_n\frac{\varepsilon_m-\varepsilon_n}{n-m}.
\end{equation}
Then
\begin{equation}\label{eq:q-expansion}
\widetilde{\mathcal Q}(b)
=\widetilde{\mathcal Q}_{\mathrm{crit}}(b)+D_S(b)+R_S(b),
\end{equation}
where $R_S$ collects the remainders in \eqref{eq:same-remainder} and \eqref{eq:cross-remainder}.

\begin{proposition}[Support-dependent perturbative estimate]\label{prop:support-dependent}
There exists a universal constant $C>0$ such that, for every finitely supported $b\in\ell^2_0(\Z)$,
\begin{align}
|D_S(b)|
&\leq C\omega_S\log(1+N)\norm{b}_2^2,
\label{eq:drift-bound}\\
|R_S(b)|
&\leq C\bigl(\eta_S+\eta_S^2\log(1+N)\bigr)\norm{b}_2^2.
\label{eq:remainder-bound}
\end{align}
Consequently,
\begin{equation}\label{eq:q-final-expansion}
\widetilde{\mathcal Q}(b)
=\norm{b}_2^2
+2\operatorname{Re}\ip{\mathcal H_{1/2}b^+}{b^-}
+E_S(b),
\end{equation}
with
\begin{equation}\label{eq:error-bound}
|E_S(b)|
\leq C\bigl(\omega_S\log(1+N)+\eta_S+\eta_S^2\log(1+N)\bigr)\norm{b}_2^2.
\end{equation}
In particular,
\begin{equation}\label{eq:q-upper}
\widetilde{\mathcal Q}(b)
\leq\left[2+C\bigl(\omega_S\log(1+N)+\eta_S+\eta_S^2\log(1+N)\bigr)\right]\norm{b}_2^2.
\end{equation}
\end{proposition}

\begin{proof}
Since $S$ is finite, the kernels
$$
K_1(m,n):=\frac{\mathbf1_{\{m\neq n\}}}{|n-m|},
\qquad
K_{1/2}(m,n):=\frac{1}{|n-m+1/2|},
$$
restricted to $S\times S$, satisfy
\begin{equation}\label{eq:row-harmonic}
\sup_{m\in S}\sum_{\substack{n\in S\\n\neq m}}\frac1{|n-m|}
\leq2\sum_{r=1}^{N-1}\frac1r
\leq C\log(1+N),
\end{equation}
and
\begin{equation}\label{eq:row-half}
\sup_{m\in S}\sum_{n\in S}\frac1{|n-m+1/2|}
\leq C\log(1+N),
\qquad
\sup_{m\in S}\sum_{n\in S}\frac1{|n-m+1/2|^2}
\leq C.
\end{equation}
The same bounds hold with $1/2$ replaced by $-1/2$. Therefore Schur's test is applicable on the finite set $S$.

Using \eqref{eq:drift}, the bound $|\varepsilon_m-\varepsilon_n|\leq\omega_S$, and \eqref{eq:row-harmonic} yields \eqref{eq:drift-bound}; the factor $s_n$ has modulus one. Likewise, \eqref{eq:same-remainder}, \eqref{eq:cross-remainder}, and \eqref{eq:row-harmonic}--\eqref{eq:row-half} give \eqref{eq:remainder-bound}. Formula \eqref{eq:q-final-expansion} follows from \eqref{eq:qcrit} and \eqref{eq:q-expansion}, and \eqref{eq:q-upper} follows from \eqref{eq:qcrit-bounds}.
\end{proof}

\begin{remark}
The logarithmic dependence on the support size is the price of the finite-support Schur estimate. A support-independent bound would require additional assumptions on $(\varepsilon_n)$ or a different operator-theoretic argument.
\end{remark}

\section{A support-independent neighborhood of structured systems}
\label{sec:structured-neighborhood}

The finite-support expansion of Section~\ref{sec:near-critical} provides first-order information on the genuine Gram form, but its error estimates contain a logarithmic factor depending on the support size. We now consider a complementary question:
the stability of an exponential system under arbitrary uniformly small perturbations of its frequencies, when the unperturbed system is already known to be a Riesz basis.

The general principle that exponential Riesz bases remain stable under sufficiently small perturbations is classical. In particular, \cite{DeCarliPathak2014} established explicit quantitative stability conditions for exponential bases and frames on multidimensional domains, including disconnected domains. Their results provide general perturbation thresholds in terms of the original Riesz bounds and the geometry of the underlying domain. The theorem below is not intended as an improvement of these general stability results.

Our emphasis is instead on the interaction between perturbation stability and the exact spectral information obtained in the preceding sections. For periodic binary patterns, the constant-matrix reduction gives the optimal Riesz bounds explicitly as the extreme eigenvalues of a finite Gram matrix. These bounds can then be used as input for a perturbation estimate around the corresponding structured exponential basis.

For this purpose, we record a direct operator-norm estimate obtained from the Taylor expansion of the exponential factors. Its proof uses only the upper Riesz bound of the reference system and yields explicit lower and upper bounds for the perturbed family. In particular, the resulting estimates are independent of the support size and remain valid for arbitrary uniformly small index-dependent perturbations, even when the perturbed frequency pattern is no longer periodic.

The contribution of this section is therefore the combination of exact pattern-dependent spectral bounds with an explicit stability neighborhood of a structured reference system, rather than a new general perturbation principle. The sufficient conditions obtained below are not claimed to be optimal. Moreover, this approach requires a reference system with a strictly positive lower Riesz bound and therefore does not resolve the support-independent near-critical problem for arbitrary nonperiodic sign configurations.
\begin{theorem}[Uniform perturbation of a structured exponential basis]\label{thm:v3-uniform-perturbation}
Let
$$
(S_0a)(x)=\sum_{n\in\Z}a_n\e^{2\pi i\lambda_n^0x}
$$
be the synthesis operator of a Riesz basis of $L^2(-1/2,1/2)$ with optimal bounds $A_0,B_0>0$. Let $(r_n)$ be real with
$$
\sup_n|r_n|\le\eta,
$$
and define
$$
(S_ra)(x)=\sum_{n\in\Z}a_n\e^{2\pi i(\lambda_n^0+r_n)x}.
$$
Then
\begin{equation}\label{eq:v3-perturbation-norm}
\norm{S_r-S_0}_{\ell^2\to L^2}
\le
\sqrt{B_0}\bigl(\e^{\pi\eta}-1\bigr).
\end{equation}
If
\begin{equation}\label{eq:v3-perturbation-condition}
\sqrt{B_0}\bigl(\e^{\pi\eta}-1\bigr)<\sqrt{A_0},
\end{equation}
then the perturbed family is a Riesz basis, with bounds
\begin{align}
A_r
&\ge
\left[
\sqrt{A_0}-\sqrt{B_0}\bigl(\e^{\pi\eta}-1\bigr)
\right]^2,
\label{eq:v3-perturbed-A}\\
B_r
&\le
B_0\e^{2\pi\eta}.
\label{eq:v3-perturbed-B}
\end{align}
\end{theorem}

\begin{proof}
For a finitely supported sequence,
$$
S_ra-S_0a
=
\sum_{k=1}^{\infty}
\frac{(2\pi ix)^k}{k!}
\sum_{n\in\Z}a_nr_n^k\e^{2\pi i\lambda_n^0x}.
$$
Since $|x|\le1/2$ and $S_0$ has upper bound $B_0$,
\begin{align*}
\norm{S_ra-S_0a}_{L^2}
&\le
\sum_{k=1}^{\infty}\frac{\pi^k}{k!}
\norm{\sum_na_nr_n^k\e^{2\pi i\lambda_n^0x}}_{L^2}\\
&\le
\sqrt{B_0}
\sum_{k=1}^{\infty}\frac{(\pi\eta)^k}{k!}\norm a_2\\
&=
\sqrt{B_0}\bigl(\e^{\pi\eta}-1\bigr)\norm a_2.
\end{align*}
Density gives \eqref{eq:v3-perturbation-norm}. The direct and reverse triangle inequalities yield \eqref{eq:v3-perturbed-A}--\eqref{eq:v3-perturbed-B}. Finally,
$$
\norm{(S_r-S_0)S_0^{-1}}
\le
\frac{\sqrt{B_0}(\e^{\pi\eta}-1)}{\sqrt{A_0}}<1,
$$
so $S_rS_0^{-1}$ is invertible by a Neumann-series argument. Hence $S_r$ is onto as well as bounded below.
\end{proof}

\begin{corollary}[A global alternating near-critical neighborhood]\label{cor:v3-critical-neighborhood}
Let
$$
\lambda_n
=
n+(-1)^n\left(\frac14-\varepsilon_n\right),
\qquad
0\le\varepsilon_n\le\eta.
$$
If
\begin{equation}\label{eq:v3-eta-condition}
\eta<\frac{\log2}{2\pi},
\end{equation}
then $\{\e^{2\pi i\lambda_nx}\}_{n\in\Z}$ is a Riesz basis. One may take
\begin{align}
A
&=
\left[
\sqrt{1-2^{-1/2}}
-
\sqrt{1+2^{-1/2}}\bigl(\e^{\pi\eta}-1\bigr)
\right]^2,
\label{eq:v3-alt-pert-A}\\
B
&=
\bigl(1+2^{-1/2}\bigr)\e^{2\pi\eta}.
\label{eq:v3-alt-pert-B}
\end{align}
\end{corollary}

\begin{proof}
At $\varepsilon_n=0$, Theorem~\ref{thm:v3-alternating} with $\delta=1/4$ gives
$$
A_0=1-2^{-1/2},
\qquad
B_0=1+2^{-1/2}.
$$
The additional perturbation is $r_n=-(-1)^n\varepsilon_n$. Condition \eqref{eq:v3-perturbation-condition} becomes
$$
\e^{\pi\eta}-1
<
\sqrt{\frac{1-2^{-1/2}}{1+2^{-1/2}}}
=
\sqrt2-1,
$$
which is equivalent to \eqref{eq:v3-eta-condition}.
\end{proof}
\begin{remark}[Comparison with general perturbation bounds]
The sufficient condition in Corollary~\ref{cor:v3-critical-neighborhood}
is obtained by combining the exact alternating Riesz
bounds with the elementary operator-norm estimate of
Theorem~\ref{thm:v3-uniform-perturbation}. It should not
be interpreted as an optimal stability radius.

Indeed, \cite{DeCarliPathak2014}
provide a general perturbation criterion which, on the
unit interval, allows a uniform perturbation of size
$\eta$ whenever
$$
\eta<
K\left(\frac{A_0}{B_0}\right),
\qquad
K(t)=\frac14-
\frac1\pi\arcsin\left(
\frac{1-\sqrt{t}}{\sqrt2}
\right).
$$
For the alternating reference system at $\delta=1/4$,
the exact bounds give
$$
\sqrt{\frac{A_0}{B_0}}=\sqrt2-1,
$$
and hence their criterion yields
$$
\eta<
\frac14-\frac1\pi\arcsin(\sqrt2-1),
$$
which is slightly less restrictive than the condition
$\eta<(\log2)/(2\pi)$ obtained above.

The role of Corollary~\ref{cor:v3-critical-neighborhood}
is therefore to illustrate how the exact spectral
information available for a structured pattern can be
combined with a simple, support-independent perturbation
estimate. No improvement over the general perturbation
thresholds of \cite{DeCarliPathak2014} is claimed.
\end{remark}

\begin{remark}\label{rem:v3-two-nearcritical-results}
The support-independent theorem and the expansion of Section~\ref{sec:near-critical} answer different questions. The former gives a global stability neighborhood around an already stable structured system. The latter retains first-order information about the separate roles of the perturbation size and its oscillation and applies directly to the genuine Gram kernel.
\end{remark}

\section{Three or more local frequency shifts}
\label{sec:finite-alphabet}

The binary decomposition extends naturally when the local perturbation takes values in a finite set. The resulting Gram operator reveals exactly what is special about the binary case: with two values there is one off-diagonal operator and one scalar invariant, whereas with three or more values one obtains a genuine block-operator matrix.

Let $\Gamma_1,\ldots,\Gamma_L$ be a partition of $\Z$, let $P_j$ be the coordinate projection onto
$$
\mathcal H_j=\ell^2(\Gamma_j),
$$
and choose real shifts $t_1,\ldots,t_L$. Define
$$
\tau_n=t_j\quad(n\in\Gamma_j),
\qquad
\lambda_n=n+\tau_n.
$$
If frequencies coincide, repetitions are retained. The Fourier-coordinate synthesis operator on $L^2(-1/2,1/2)$ is
\begin{equation}\label{eq:v3-finite-alphabet-N}
N=\sum_{j=1}^L T_{t_j}P_j.
\end{equation}
Since each $T_{t_j}$ is unitary,
$$
\norm{Na}
\le
\sum_{j=1}^L\norm{P_ja}
\le
\sqrt L\norm a.
$$

\begin{proposition}[Block Gram matrix]\label{prop:v3-block-gram-L}
Relative to
$$
\ell^2(\Z)=\bigoplus_{j=1}^L\mathcal H_j,
$$
the Gram operator $G=N^*N$ has blocks
\begin{equation}\label{eq:v3-Gjk}
G_{jk}
=
P_jT_{t_k-t_j}P_k:
\mathcal H_k\longrightarrow\mathcal H_j.
\end{equation}
In particular,
$$
G_{jj}=I_{\mathcal H_j},
\qquad
G_{kj}=G_{jk}^*.
$$
Define the real symmetric matrix $R=(r_{jk})_{j,k=1}^L$ by
$$
r_{jj}=0,
\qquad
r_{jk}=\norm{G_{jk}}\quad(j\ne k),
\qquad
\eta=\norm R_{\C^L\to\C^L}.
$$
Then
\begin{equation}\label{eq:v3-L-bound}
(1-\eta)\norm a^2
\le
\norm{Na}^2
\le
(1+\eta)\norm a^2.
\end{equation}
Hence $\eta<1$ is sufficient for the associated family to be a Riesz sequence.
\end{proposition}

\begin{proof}
From \eqref{eq:v3-finite-alphabet-N},
$$
N^*N
=
\sum_{j,k=1}^L
P_jT_{-t_j}T_{t_k}P_k
=
\sum_{j,k=1}^L
P_jT_{t_k-t_j}P_k,
$$
which gives the block identity and the diagonal and adjoint relations.

Write $a_j=P_ja$ and
$$
v=(\norm{a_1},\ldots,\norm{a_L})^T.
$$
Then
\begin{align*}
\left|\ip{(G-I)a}{a}\right|
&\le
\sum_{j\ne k}\norm{G_{jk}}\norm{a_k}\norm{a_j}\\
&=
v^TRv\\
&\le
\norm R\norm v^2
=
\eta\norm a^2.
\end{align*}
Since $\ip{Ga}{a}=\norm{Na}^2$, inequality \eqref{eq:v3-L-bound} follows.
\end{proof}

\begin{remark}\label{rem:v3-L-not-optimal}
The comparison bound \eqref{eq:v3-L-bound} is sufficient and need not be optimal. For a general nonperiodic finite partition it yields a lower synthesis estimate when $\eta<1$, but it does not by itself provide completeness. In the binary case $L=2$, the self-adjoint off-diagonal part has the special form
$$
\begin{pmatrix}0&C^*\\C&0\end{pmatrix},
$$
so one norm $\norm C$ gives the exact endpoints. For $L\ge3$, no analogous one-scalar reduction is available in general.
\end{remark}

\subsection{Periodic finite-alphabet patterns}

Assume now that the local shift sequence is $p$-periodic:
$$
\tau_{n+p}=\tau_n.
$$
Write $\tau_r$ for the shift carried by residue class $r$, $0\le r\le p-1$. Define
\begin{equation}\label{eq:v3-Htau}
H_{p,\tau}
=
\left(
\frac1{\sqrt p}\e^{-2\pi ij(r+\tau_r)/p}
\right)_{j,r=0}^{p-1},
\end{equation}
\begin{equation}\label{eq:v3-Etau}
E_\tau(\xi)
=
\operatorname{diag}\left(\e^{-2\pi i\tau_r\xi}\right)_{r=0}^{p-1}.
\end{equation}

\begin{theorem}[Exact periodic reduction for finitely many local shifts]\label{thm:v3-periodic-finite-alphabet}
The period-$p$ synthesis symbol satisfies
\begin{equation}\label{eq:v3-finite-alphabet-factorization}
\mathcal N_{p,\tau}(\xi)
=
\sum_{r=0}^{p-1}D_{\tau_r}(\xi)Q_r
=
H_{p,\tau}E_\tau(\xi)V^*.
\end{equation}
The optimal squared synthesis bounds are
\begin{equation}\label{eq:v3-finite-alphabet-exact-bounds}
A=\lambda_{\min}(H_{p,\tau}^*H_{p,\tau}),
\qquad
B=\lambda_{\max}(H_{p,\tau}^*H_{p,\tau}).
\end{equation}
The exponential family is a Riesz basis if and only if
\begin{equation}\label{eq:v3-finite-alphabet-collision}
r+\tau_r\not\equiv q+\tau_q\pmod p
\qquad(0\le q<r\le p-1).
\end{equation}
\end{theorem}

\begin{proof}
The computation of Theorem~\ref{thm:v3-constant-binary} does not use the fact that the shifts take only two values. The $j$-th entry of the $r$-th column of $\mathcal N_{p,\tau}(\xi)V$ is
$$
\frac1{\sqrt p}
\e^{-2\pi i\tau_r(\xi+j/p)}
\e^{-2\pi ijr/p}
=
\frac1{\sqrt p}
\e^{-2\pi ij(r+\tau_r)/p}
\e^{-2\pi i\tau_r\xi}.
$$
This proves \eqref{eq:v3-finite-alphabet-factorization}. Since the two right-hand factors are unitary, the sharp norm argument gives \eqref{eq:v3-finite-alphabet-exact-bounds}.

Set
$$
z_r=\e^{-2\pi i(r+\tau_r)/p}.
$$
Then
$$
\det H_{p,\tau}
=
p^{-p/2}
\prod_{0\le q<r\le p-1}(z_r-z_q),
$$
so $H_{p,\tau}$ is invertible exactly under \eqref{eq:v3-finite-alphabet-collision}. Uniform invertibility of the fiber symbol follows from
$$
\mathcal N_{p,\tau}(\xi)^{-1}
=
V E_\tau(\xi)^*H_{p,\tau}^{-1}.
$$
If the determinant vanishes, a nonzero vector in $\ker H_{p,\tau}$ produces a nonzero kernel vector for synthesis after the inverse fiber transformation.
\end{proof}

\begin{corollary}[A sufficient oscillation condition]\label{cor:v3-shift-oscillation}
Every fixed periodic local-shift pattern satisfying
\begin{equation}\label{eq:v3-oscillation-condition}
\max_{0\le r<p}\tau_r-
\min_{0\le r<p}\tau_r<1
\end{equation}
generates a Riesz basis of $L^2(-1/2,1/2)$.
\end{corollary}

\begin{proof}
If a collision occurred, then
$$
r-q+\tau_r-\tau_q\in p\Z.
$$
Hence $\tau_r-\tau_q$ would be an integer. Under \eqref{eq:v3-oscillation-condition} its absolute value is strictly less than one, so it must be zero. This would imply $r-q\in p\Z$, impossible for distinct residues.
\end{proof}

\begin{remark}\label{rem:v3-common-shift}
The criterion concerns the range of the local shifts. Adding a common constant to every $\tau_r$ multiplies every exponential by the same unimodular function and therefore leaves all Gram eigenvalues unchanged.
\end{remark}

\subsection{A three-shift example without binary spectral symmetry}

Consider the period-three pattern
$$
\tau_{3k}=\frac14,
\qquad
\tau_{3k+1}=0,
\qquad
\tau_{3k+2}=-\frac14.
$$
Its frequencies in one period are
$$
3k+\frac14,
\qquad
3k+1,
\qquad
3k+\frac74.
$$
Subtracting the common shift $1/4$ preserves the Gram eigenvalues and reduces the constant matrix, up to a left unitary diagonal factor, to
$$
\widetilde H
=
\frac1{\sqrt3}
\begin{pmatrix}
1&1&1\\
1&-i&-1\\
1&-1&1
\end{pmatrix}.
$$
A direct multiplication gives
\begin{equation}\label{eq:v3-three-shift-gram}
\widetilde H^*\widetilde H
=
\begin{pmatrix}
1&-i/3&1/3\\
i/3&1&-i/3\\
1/3&i/3&1
\end{pmatrix}
=
\frac43 I_3-\frac13ww^*,
\qquad
w=
\begin{pmatrix}
1\\-i\\-1
\end{pmatrix}.
\end{equation}
Since $\norm w^2=3$, the eigenvalue on $\operatorname{span}\{w\}$ is $1/3$, while both eigenvalues on $w^\perp$ are $4/3$. Therefore the family is a Riesz basis with exact squared bounds
\begin{equation}\label{eq:v3-three-shift-bounds}
A=\frac13,
\qquad
B=\frac43.
\end{equation}
In particular,
$$
\sigma(\widetilde H^*\widetilde H-I)
=
\left\{-\frac23,\frac13\right\},
$$
which is not symmetric about zero. Thus the exact binary form $1\pm\rho$ does not extend to a general three-shift pattern; the correct periodic invariant is the full finite Gram matrix $H_{p,\tau}^*H_{p,\tau}$.

\section{Two shifted lattices on two intervals}
\label{sec:twointervals}

\subsection{Integer translation}\label{sec:integer-twointerval}

We now turn from index-dependent local perturbations to the disconnected-domain branch of the paper. The simplest nontrivial configuration consists of two complete frequency lattices acting on a union of two unit intervals. In this setting the synthesis operator reduces to an explicit $2\times2$ matrix field, providing a concrete model for the general finite-cell fiberization developed in the next section. Here the frequency set is the union of two complete lattices, $\Z$ and $\Z+\delta$, while the spatial domain is the union of two intervals. The resulting synthesis map has two channels and is naturally represented by a $2\times2$ matrix.

Let $\alpha\in\Z$ with $|\alpha|\geq1$, let $0<\delta<1$, and define $E=[0,1]\cup[\alpha,\alpha+1]$.
For finitely supported sequences $a=(a_n)$ and $b=(b_n)$, set
$$A(x)=\sum_{n\in\Z}a_n\e^{2\pi i n x},\quad B_\delta(x)=\sum_{n\in\Z}b_n\e^{2\pi i(n+\delta)x},\quad F(x)=A(x)+B_\delta(x).$$
\begin{theorem}[Exact bounds and the integer-translation criterion]\label{thm:integer-twointerval}
For every finitely supported $a,b\in\ell^2(\Z)$,
\begin{align}
&2\bigl(1-|\cos(\pi\alpha\delta)|\bigr)
\bigl(\norm{a}_2^2+\norm{b}_2^2\bigr)
\leq\norm{F}_{L^2(E)}^2\notag\\
&\qquad\leq
2\bigl(1+|\cos(\pi\alpha\delta)|\bigr)
\bigl(\norm{a}_2^2+\norm{b}_2^2\bigr).
\label{eq:integer-bounds}
\end{align}
The constants are optimal. Moreover,
\begin{equation}\label{eq:integer-basis-criterion}
\{\e^{2\pi i n x}\}_{n\in\Z}
\cup
\{\e^{2\pi i(n+\delta)x}\}_{n\in\Z}
\text{ is a Riesz basis of }L^2(E)
\quad\Longleftrightarrow\quad
\alpha\delta\notin\Z.
\end{equation}
\end{theorem}

\begin{proof}
Since $\alpha\in\Z$,
$$A(x+\alpha)=\sum_n a_n\e^{2\pi i n(x+\alpha)}=\sum_n a_n\e^{2\pi i n x}\e^{2\pi i n\alpha}=\sum_n a_n\e^{2\pi i n x}=A(x),$$
because $n\alpha\in\Z$. Similarly,
$$B_\delta(x+\alpha)=\sum_n b_n\e^{2\pi i(n+\delta)(x+\alpha)}=\sum_n b_n\e^{2\pi i(n+\delta)x}
\e^{2\pi i n\alpha}\e^{2\pi i\alpha\delta}=\e^{2\pi i\alpha\delta}B_\delta(x).$$
Let $\vartheta=2\pi\alpha\delta$. For each $x\in[0,1]$, the two values of $F$ on the two copies of the unit interval are
\begin{align*}
F(x)&=A(x)+B_\delta(x),\\
F(x+\alpha)&=A(x)+\e^{i\vartheta}B_\delta(x).
\end{align*}
Writing these equations as a vector identity gives
\begin{equation}\label{eq:M-vartheta-representation}
\begin{pmatrix}
F(x)\\[1mm]
F(x+\alpha)
\end{pmatrix}
=
\underbrace{
\begin{pmatrix}
1&1\\
1&\e^{i\vartheta}
\end{pmatrix}}_{M_\vartheta}
\begin{pmatrix}
A(x)\\[1mm]
B_\delta(x)
\end{pmatrix}.
\end{equation}
The first column $(1,1)^T$ records how $A$ appears on the two intervals, while the second column $(1,\e^{i\vartheta})^T$ records how $B_\delta$ appears.

Define
$$
z(x)=\begin{pmatrix}A(x)\\B_\delta(x)\end{pmatrix}.
$$
Since the intervals are disjoint up to endpoints,
\begin{align}
\norm{F}_{L^2(E)}^2
&=\int_0^1|F(x)|^2\dd x
+\int_\alpha^{\alpha+1}|F(y)|^2\dd y=\int_0^1|F(x)|^2\dd x
+\int_0^1|F(x+\alpha)|^2\dd x\\
&=\int_0^1\bigl(|F(x)|^2+|F(x+\alpha)|^2\bigr)\dd x=\int_0^1\norm{M_\vartheta z(x)}_{\C^2}^2\dd x.
\label{eq:integer-fiber-norm}
\end{align}
For any complex matrix $M$ and vector $z$,
\begin{equation}\label{eq:matrix-norm-gram}
\norm{Mz}_{\C^2}^2=(Mz)^*(Mz)=z^*M^*Mz.
\end{equation}
Thus the fiber Gram matrix is $H_\vartheta=M_\vartheta^*M_\vartheta$. Computing it entry by entry,
$$M_\vartheta^*
=\begin{pmatrix}
1&1\\
1&\e^{-i\vartheta}
\end{pmatrix},\quad 
H_\vartheta
=\begin{pmatrix}
1&1\\
1&\e^{-i\vartheta}
\end{pmatrix}
\begin{pmatrix}
1&1\\
1&\e^{i\vartheta}
\end{pmatrix}=\begin{pmatrix}
2&1+\e^{i\vartheta}\\
1+\e^{-i\vartheta}&2
\end{pmatrix}.$$
Consequently,
\begin{equation}\label{eq:integer-gram-integral}
\norm{F}_{L^2(E)}^2
=\int_0^1z(x)^*H_\vartheta z(x)\dd x.
\end{equation}
To find the best pointwise constants, compute the eigenvalues of $H_\vartheta$. They are the roots of
\begin{equation}
0=\det(H_\vartheta-\lambda I)=\det\begin{pmatrix}
2-\lambda&1+\e^{i\vartheta}\\
1+\e^{-i\vartheta}&2-\lambda
\end{pmatrix}=(2-\lambda)^2-(1+\e^{i\vartheta})(1+\e^{-i\vartheta}).
\label{eq:integer-characteristic}
\end{equation}
The product in the second term is
$$(1+\e^{i\vartheta})(1+\e^{-i\vartheta})=1+\e^{i\vartheta}+\e^{-i\vartheta}+1=2+2\cos\vartheta=4\cos^2\left(\frac\vartheta2\right).$$
Therefore
$$
(2-\lambda)^2=4\cos^2\left(\frac\vartheta2\right),
$$
and the two eigenvalues are
\begin{equation}
\lambda_-=2-2\left|\cos\left(\frac\vartheta2\right)\right|,\quad
\lambda_+=2+2\left|\cos\left(\frac\vartheta2\right)\right|.
\label{eq:integer-eigenvalues-vartheta}
\end{equation}
Since $\vartheta/2=\pi\alpha\delta$,
\begin{equation}\label{eq:integer-eigenvalues}
\lambda_\pm=2\pm2|\cos(\pi\alpha\delta)|.
\end{equation}
For every $z\in\C^2$, the spectral theorem gives
\begin{equation}\label{eq:integer-pointwise-spectral}
\lambda_-\norm{z}_{\C^2}^2
\leq z^*H_\vartheta z
\leq\lambda_+\norm{z}_{\C^2}^2.
\end{equation}
Applying \eqref{eq:integer-pointwise-spectral} to $z=z(x)$ and integrating,
\begin{equation}
\lambda_-\int_0^1\bigl(|A(x)|^2+|B_\delta(x)|^2\bigr)\dd x
\leq\norm{F}_{L^2(E)}^2\leq\lambda_+\int_0^1\bigl(|A(x)|^2+|B_\delta(x)|^2\bigr)\dd x.
\label{eq:integer-integrated-spectral}
\end{equation}
Parseval gives
\begin{equation}\label{eq:A-parseval}
\int_0^1|A(x)|^2\dd x=\sum_n|a_n|^2=\norm{a}_2^2,
\end{equation}
and
\begin{equation}
\int_0^1|B_\delta(x)|^2\dd x=\int_0^1\left|\e^{2\pi i\delta x}\sum_n b_n\e^{2\pi i n x}\right|^2\dd x=\int_0^1\left|\sum_n b_n\e^{2\pi i n x}\right|^2\dd x=\sum_n|b_n|^2=\norm{b}_2^2.
\label{eq:Bdelta-parseval}
\end{equation}
Substituting \eqref{eq:integer-eigenvalues}, \eqref{eq:A-parseval}, and \eqref{eq:Bdelta-parseval} into \eqref{eq:integer-integrated-spectral} gives \eqref{eq:integer-bounds}.

The constants are optimal because $A$ and $B_\delta$ can be chosen independently as arbitrary $L^2(0,1)$ functions: multiplication by $\e^{2\pi i\delta x}$ is unitary, and both ordinary exponential systems are orthonormal bases of $L^2(0,1)$. Thus $z(x)$ may be chosen in an eigenvector direction of $H_\vartheta$. If one first restricts to finite sequences, the corresponding functions can be approximated in $L^2$ by finite Fourier sums.

It remains to characterize the basis property. The determinant of the coefficient matrix is
\begin{equation}\label{eq:integer-det}
\det M_\vartheta=\e^{i\vartheta}-1.
\end{equation}
Thus
\begin{align}
M_\vartheta\text{ is invertible}
&\Longleftrightarrow\e^{i\vartheta}\neq1\\
&\Longleftrightarrow\vartheta\notin2\pi\Z\\
&\Longleftrightarrow\alpha\delta\notin\Z.
\label{eq:integer-invertible}
\end{align}
This is equivalent to $\cos(\pi\alpha\delta)\neq\pm1$.

Assume first that $\alpha\delta\notin\Z$. Given an arbitrary $g\in L^2(E)$, define
\begin{equation}\label{eq:integer-unfold}
y(x)=\begin{pmatrix}g(x)\\g(x+\alpha)\end{pmatrix},
\qquad 0\leq x<1.
\end{equation}
Then
\begin{equation}\label{eq:integer-inverse-fiber}
z(x)=M_\vartheta^{-1}y(x)
\end{equation}
belongs to $L^2(0,1;\C^2)$. Write $z=(A,B_\delta)^T$, define
$$
B(x)=\e^{-2\pi i\delta x}B_\delta(x),
$$
and take the ordinary Fourier coefficients of $A$ and $B$. This produces $a,b\in\ell^2(\Z)$ whose synthesis equals $g$. Hence the synthesis operator is onto; together with \eqref{eq:integer-bounds}, this proves the Riesz-basis assertion.

Conversely, suppose $\alpha\delta\in\Z$. Then $\e^{i\vartheta}=1$ and the two columns of $M_\vartheta$ coincide. Choose any nonzero $h\in L^2(0,1)$ and set
\begin{equation}\label{eq:integer-null-choice}
A=h,
\qquad
B_\delta=-h.
\end{equation}
Define $B(x)=\e^{-2\pi i\delta x}B_\delta(x)$ and let $a,b\in\ell^2(\Z)$ be the Fourier coefficient sequences of $A$ and $B$. Then
$$
F(x)=A(x)+B_\delta(x)=0
$$
on the first interval, while on the second interval
$$
F(x+\alpha)=A(x)+\e^{i\vartheta}B_\delta(x)=h(x)-h(x)=0.
$$
The coefficient pair is nonzero. Finite truncations of the coefficient sequences approximate this nonzero kernel vector and show that no positive lower Riesz bound can hold. This proves the converse in \eqref{eq:integer-basis-criterion}.
\end{proof}

\subsection{Non-integer translation}\label{sec:noninteger-twointerval}

Exponential bases on disconnected intervals have been constructed by several methods, including perturbations of integer frequencies subject to conditions depending on the gaps between the intervals; see \cite{AsipchukDrezels2023}. The related question of preserving exponential bases under modifications of the interval positions has been studied in \cite{Asipchuk2026}. Our problem is different in that the frequency set is fixed to the union of two complete shifted lattices, $\mathbb Z\cup(\mathbb Z+\delta)$, while the relative position of the two intervals is allowed to vary. We seek necessary and sufficient conditions for the Riesz-basis property, together with the optimal bounds. For a non-integer translation, folding modulo one produces two distinct matrix fibers, and stability requires both of them to be invertible.

Let $0<\delta<1$, let $\alpha\in\R\setminus\Z$, and assume that $[0,1]$ and $[\alpha,\alpha+1]$ are disjoint. Write
\begin{equation}\label{eq:alpha-q-beta}
\alpha=q+\beta,
\qquad q=\lfloor\alpha\rfloor,
\qquad 0<\beta<1,
\end{equation}
and set $E=[0,1]\cup[\alpha,\alpha+1]$. For finitely supported $a=(a_n)$ and $b=(b_n)$ define
\begin{equation}\label{eq:F-noninteger}
F(x)=\sum_{n\in\Z}a_n\e^{2\pi i n x}
+\sum_{n\in\Z}b_n\e^{2\pi i(n+\delta)x}.
\end{equation}
Let
\begin{equation}\label{eq:c-alpha-delta}
c_{\alpha,\delta}
=2\max\bigl\{|\cos(\pi q\delta)|,
|\cos(\pi(q+1)\delta)|\bigr\}.
\end{equation}

\begin{theorem}[Exact bounds and the non-integer criterion]\label{thm:noninteger-twointerval}
For every finitely supported $a,b\in\ell^2(\Z)$,
\begin{equation}
(2-c_{\alpha,\delta})\bigl(\norm{a}_2^2+\norm{b}_2^2\bigr)\leq\norm{F}_{L^2(E)}^2\leq(2+c_{\alpha,\delta})\bigl(\norm{a}_2^2+\norm{b}_2^2\bigr).
\label{eq:noninteger-bounds}
\end{equation}
The constants are optimal. Moreover,
\begin{align}
&\{\e^{2\pi i n x}\}_{n\in\Z}
\cup\{\e^{2\pi i(n+\delta)x}\}_{n\in\Z}
\text{ is a Riesz basis of }L^2(E)\notag\\
&\qquad\Longleftrightarrow
q\delta\notin\Z
\quad\text{and}\quad
(q+1)\delta\notin\Z.
\label{eq:noninteger-basis-criterion}
\end{align}
\end{theorem}

\begin{proof}
Write
\begin{equation}\label{eq:ABC-noninteger}
A(x)=\sum_n a_n\e^{2\pi i n x},
\qquad
B(x)=\sum_n b_n\e^{2\pi i n x},
\qquad
C(x)=\e^{2\pi i\delta x}B(x).
\end{equation}
Then
\begin{equation}\label{eq:F-A-C}
F(x)=A(x)+C(x).
\end{equation}
Both $A$ and $B$ are one-periodic.

The second interval is
$$
[\alpha,\alpha+1]=[q+\beta,q+\beta+1].
$$
Split it as
\begin{equation}\label{eq:split-second-interval}
[q+\beta,q+1)\cup[q+1,q+\beta+1].
\end{equation}
The first piece is parameterized by $y=x+q$, with $\beta\leq x<1$, and the second by $y=x+q+1$, with $0\leq x<\beta$. Hence
\begin{equation}\label{eq:second-integral-split}
\int_\alpha^{\alpha+1}|F(y)|^2\dd y
=\int_\beta^1|F(x+q)|^2\dd x
+\int_0^\beta|F(x+q+1)|^2\dd x.
\end{equation}
For an arbitrary integer $k$, periodicity gives
$$A(x+k)=A(x),\quad
C(x+k)=\e^{2\pi i\delta(x+k)}B(x+k)=\e^{2\pi i k\delta}\e^{2\pi i\delta x}B(x)=\e^{2\pi i k\delta}C(x).$$
Therefore
\begin{align}
F(x+q+1)
&=A(x)+\e^{2\pi i(q+1)\delta}C(x),
&&0\leq x<\beta,
\label{eq:F-qplus1}\\
F(x+q)
&=A(x)+\e^{2\pi i q\delta}C(x),
&&\beta\leq x<1.
\label{eq:F-q}
\end{align}
Combining \eqref{eq:F-A-C}, \eqref{eq:F-qplus1}, and \eqref{eq:F-q}, define
\begin{equation}\label{eq:M-k}
M_k=\begin{pmatrix}
1&1\\
1&\e^{2\pi i k\delta}
\end{pmatrix}.
\end{equation}
Then
\begin{equation}\label{eq:piece-matrix-qplus1}
\begin{pmatrix}
F(x)\\F(x+q+1)
\end{pmatrix}
=M_{q+1}\begin{pmatrix}A(x)\\C(x)\end{pmatrix},
\qquad 0\leq x<\beta,
\end{equation}
and
\begin{equation}\label{eq:piece-matrix-q}
\begin{pmatrix}
F(x)\\F(x+q)
\end{pmatrix}
=M_q\begin{pmatrix}A(x)\\C(x)\end{pmatrix},
\qquad \beta\leq x<1.
\end{equation}
Thus, with $z(x)=(A(x),C(x))^T$,
\begin{equation}\label{eq:noninteger-fiber-norm}
\norm{F}_{L^2(E)}^2
=\int_0^\beta\norm{M_{q+1}z(x)}_{\C^2}^2\dd x
+\int_\beta^1\norm{M_qz(x)}_{\C^2}^2\dd x.
\end{equation}

For each integer $k$, the Gram matrix is
\begin{equation}\label{eq:H-k}
H_k=M_k^*M_k
=\begin{pmatrix}
2&1+\e^{2\pi i k\delta}\\
1+\e^{-2\pi i k\delta}&2
\end{pmatrix}.
\end{equation}
Repeating the determinant calculation in \eqref{eq:integer-characteristic}--\eqref{eq:integer-eigenvalues-vartheta}, now with $\vartheta=2\pi k\delta$, gives
\begin{equation}\label{eq:H-k-eigenvalues}
\lambda_\pm(H_k)=2\pm2|\cos(\pi k\delta)|.
\end{equation}
Consequently, on $[0,\beta)$,
\begin{equation}
\bigl(2-2|\cos(\pi(q+1)\delta)|\bigr)\norm{z(x)}^2
\leq\norm{M_{q+1}z(x)}^2\leq\bigl(2+2|\cos(\pi(q+1)\delta)|\bigr)\norm{z(x)}^2,
\label{eq:qplus1-pointwise}
\end{equation}
whereas on $[\beta,1)$,
\begin{equation}
\bigl(2-2|\cos(\pi q\delta)|\bigr)\norm{z(x)}^2\leq\norm{M_qz(x)}^2\leq\bigl(2+2|\cos(\pi q\delta)|\bigr)\norm{z(x)}^2.
\label{eq:q-pointwise}
\end{equation}
The global lower constant is the smaller of the two lower constants:
\begin{align*}
&\min\bigl\{2-2|\cos(\pi q\delta)|,
2-2|\cos(\pi(q+1)\delta)|\bigr\}\\
&\qquad=2-2\max\bigl\{|\cos(\pi q\delta)|,
|\cos(\pi(q+1)\delta)|\bigr\}=2-c_{\alpha,\delta}.
\end{align*}
Similarly, the global upper constant is
\begin{align*}
&\max\bigl\{2+2|\cos(\pi q\delta)|,
2+2|\cos(\pi(q+1)\delta)|\bigr\}\\
&\qquad=2+2\max\bigl\{|\cos(\pi q\delta)|,
|\cos(\pi(q+1)\delta)|\bigr\}=2+c_{\alpha,\delta}.
\end{align*}
Integrating \eqref{eq:qplus1-pointwise}--\eqref{eq:q-pointwise} and using
$$\int_0^1\norm{z(x)}_{\C^2}^2\dd x=\int_0^1\bigl(|A(x)|^2+|C(x)|^2\bigr)\dd x=\norm{a}_2^2+\norm{b}_2^2$$
gives \eqref{eq:noninteger-bounds}.

The constants do not depend on $\beta$. Even if one of the two pieces has small but positive length, an $L^2$ vector function may be supported entirely on that piece, and the matrix on that piece can determine the global lower or upper bound. Choosing an eigenvector direction on the relevant piece, followed by finite Fourier approximation, also proves optimality.

We now prove the basis criterion. First suppose
\begin{equation}\label{eq:noninteger-condition}
q\delta\notin\Z
\qquad\text{and}\qquad
(q+1)\delta\notin\Z.
\end{equation}
Then both matrices $M_q$ and $M_{q+1}$ are invertible. The piecewise matrix
\begin{equation}\label{eq:piecewise-M}
M(x)=
\begin{cases}
M_{q+1},&0\leq x<\beta,\\
M_q,&\beta\leq x<1
\end{cases}
\end{equation}
has an essentially bounded inverse because it assumes only two invertible matrix values. Given any $g\in L^2(E)$, unfold it into
\begin{equation}\label{eq:noninteger-unfold}
y(x)=
\begin{cases}
(g(x),g(x+q+1))^T,&0\leq x<\beta,\\
(g(x),g(x+q))^T,&\beta\leq x<1.
\end{cases}
\end{equation}
Set $z(x)=M(x)^{-1}y(x)$. Then $z=(A,C)^T\in L^2(0,1;\C^2)$. Defining $B(x)=\e^{-2\pi i\delta x}C(x)
$ and taking the Fourier coefficients of $A$ and $B$ gives sequences $a,b\in\ell^2(\Z)$ whose synthesis equals $g$. Thus the synthesis map is onto, and \eqref{eq:noninteger-bounds} gives the two-sided estimate.

Conversely, suppose first that $q\delta\in\Z$. Then $\e^{2\pi i q\delta}=1$, and the matrix on $[\beta,1)$ becomes
\begin{equation}\label{eq:Mq-singular}
M_q=\begin{pmatrix}1&1\\1&1\end{pmatrix}.
\end{equation}
This matrix has the nonzero null vector $(1,-1)^T$. Choose a nonzero function
\begin{equation}\label{eq:h-support-q}
h\in L^2(0,1),
\qquad \operatorname{supp}h\subseteq[\beta,1),
\end{equation}
and set
\begin{equation}\label{eq:A-C-null}
A(x)=h(x),
\qquad
C(x)=-h(x).
\end{equation}
On $[\beta,1)$,
\begin{align}
F(x)&=A(x)+C(x)=h(x)-h(x)=0,\notag\\
F(x+q)&=A(x)+\e^{2\pi i q\delta}C(x)
=h(x)-h(x)=0.
\label{eq:null-on-q-piece}
\end{align}
On $[0,\beta)$, both $A$ and $C$ vanish by the support condition, so the two corresponding values of $F$ also vanish. Thus $F=0$ almost everywhere on $E$.

To translate this construction into coefficient sequences, define again $B(x)=\e^{-2\pi i\delta x}C(x)$.
Since multiplication by $\e^{-2\pi i\delta x}$ is unitary on $L^2(0,1)$, both $A$ and $B$ belong to $L^2(0,1)$. Let
\begin{align}
a_n&=\int_0^1A(x)\e^{-2\pi i n x}\dd x,\notag\\
b_n&=\int_0^1B(x)\e^{-2\pi i n x}\dd x.
\label{eq:null-coefficients}
\end{align}
Parseval gives $a,b\in\ell^2(\Z)$, and they are not both zero because $h\neq0$. Hence the synthesis operator has a nontrivial kernel. Finite truncations again show that no positive lower Riesz bound can hold.

If instead $(q+1)\delta\in\Z$, use exactly the same construction with a nonzero $h$ supported in $[0,\beta)$. On that interval, $M_{q+1}$ has identical columns, and the choice $A=h$, $C=-h$ again produces a nonzero coefficient pair whose synthesis is zero. This completes the converse and proves \eqref{eq:noninteger-basis-criterion}.
\end{proof}

\section{Matrix fiberization on finite unions of intervals}
\label{sec:finite-unions}

The two-interval calculations of the preceding section are instances of a more general folding principle. We first extend this mechanism to finitely many translated unit intervals and then to arbitrary finite unions of bounded intervals, where the target fiber dimension may vary from cell to cell.

The fiberization of exponential systems over lattice multi-tiles and sub-multi-tiles is well established. In particular, \cite{AgoraAntezanaCabrelli2015} obtained Riesz-basis results for lattice multi-tiles using finite-dimensional fiber matrices, together with existence and converse multi-tiling theorems. \cite{BarbieriEtAl2018} subsequently treated exponential frames on sub-multi-tiles and gave exact frame criteria and optimal bounds in terms of possibly rectangular fiber matrices. Related explicit Fourier-matrix criteria for unions of translated cubes were also studied in \cite{AsipchukDeCarliLi2025}.

Our purpose in this section is not to introduce a new general fiberization principle, but to develop its explicit finite-cell realization for finite unions of intervals. In this geometry, the fiber changes are determined directly by the integer and fractional parts of the interval endpoints. This permits a concrete description of the local phase matrices, their exact spectral bounds, and their dependence on endpoint configurations.

We retain the proofs of the rank and spectral criteria in this specialized setting in order to make the subsequent calculations and applications self-contained. The resulting framework also provides the notation used for the arithmetic fiber-exchange arguments and the commensurate-lattice reduction below.

\subsection{Finitely many translated unit intervals}
\label{subsec:v3-unit-multiinterval}

Let
\begin{equation}\label{eq:v3-unit-E}
E=\bigcup_{j=1}^J[\alpha_j,\alpha_j+1),
\end{equation}
where the intervals are pairwise disjoint up to endpoints, and write
\begin{equation}\label{eq:v3-unit-alpha}
\alpha_j=q_j+\beta_j,
\qquad
q_j\in\Z,
\qquad
0\le\beta_j<1.
\end{equation}
Let $\delta_1,\ldots,\delta_K$ be pairwise incongruent modulo $\Z$, and define
\begin{equation}\label{eq:v3-unit-Lambda}
\Lambda=\bigcup_{k=1}^K(\Z+\delta_k).
\end{equation}
For coefficient sequences $a^{(k)}=(a_n^{(k)})_{n\in\Z}$, set
\begin{equation}\label{eq:v3-unit-synthesis}
(\mathcal Sa)(y)
=
\sum_{k=1}^K\sum_{n\in\Z}
a_n^{(k)}\e^{2\pi i(n+\delta_k)y}.
\end{equation}

For $x\in[0,1)$, the point of the $j$-th interval having fractional part $x$ is
\begin{equation}\label{eq:v3-unit-yj}
y_j(x)=x+\kappa_j(x),
\qquad
\kappa_j(x)=q_j+\mathbf1_{[0,\beta_j)}(x).
\end{equation}
For each frequency channel put
\begin{equation}\label{eq:v3-unit-Ak}
B_k(x)=\sum_{n\in\Z}a_n^{(k)}\e^{2\pi inx},
\qquad
A_k(x)=\e^{2\pi i\delta_kx}B_k(x),
\end{equation}
and define the $J\times K$ matrix field
\begin{equation}\label{eq:v3-unit-Mx}
M(x)
=
\left(
\e^{2\pi i\delta_k\kappa_j(x)}
\right)_{
\substack{1\le j\le J\\1\le k\le K}
}.
\end{equation}

\begin{theorem}[Exact matrix-symbol fiberization for unit intervals]\label{thm:v3-unit-fiberization}
With the preceding notation,
\begin{equation}\label{eq:v3-unit-fiber-identity}
\norm{\mathcal Sa}_{L^2(E)}^2
=
\int_0^1\norm{M(x)A(x)}_{\C^J}^2\dd x,
\end{equation}
where $A=(A_1,\ldots,A_K)^T$. Moreover,
\begin{equation}\label{eq:v3-unit-coeff-identity}
\int_0^1\norm{A(x)}_{\C^K}^2\dd x
=
\sum_{k=1}^K\norm{a^{(k)}}_{\ell^2}^2.
\end{equation}
Every vector function in $L^2(0,1;\C^K)$ arises from a unique coefficient family.
\end{theorem}

\begin{proof}
At $y_j(x)=x+\kappa_j(x)$, periodicity of $B_k$ gives
\begin{align*}
(\mathcal Sa)(y_j(x))
&=
\sum_{k=1}^K
\e^{2\pi i\delta_k(x+\kappa_j(x))}
B_k(x+\kappa_j(x))\\
&=
\sum_{k=1}^K
\e^{2\pi i\delta_k\kappa_j(x)}A_k(x).
\end{align*}
Thus the vector of folded values is $M(x)A(x)$. Changing variables on each disjoint interval gives \eqref{eq:v3-unit-fiber-identity}. Parseval and $|\e^{2\pi i\delta_kx}|=1$ give \eqref{eq:v3-unit-coeff-identity}.

Conversely, given $A_k\in L^2(0,1)$, define
$$
B_k(x)=\e^{-2\pi i\delta_kx}A_k(x)
$$
and take its ordinary Fourier coefficients. This produces the unique sequence $a^{(k)}\in\ell^2(\Z)$.
\end{proof}

The dependence of $M(x)$ on $x$ is extremely simple. List the distinct points of
$$
\{0,1,\beta_1,\ldots,\beta_J\}
$$
as
$$
0=\gamma_0<\gamma_1<\cdots<\gamma_R=1
$$
and put
$$
I_r=(\gamma_{r-1},\gamma_r),
\qquad 1\le r\le R.
$$
Each $\kappa_j$ is constant on every $I_r$, and hence
\begin{equation}\label{eq:v3-unit-Mr}
M(x)=M_r,
\qquad x\in I_r,
\end{equation}
for finitely many matrices $M_r\in\C^{J\times K}$.

\begin{lemma}[Uniform bounds from finitely many matrices]\label{lem:v3-finite-cell}
Let
$$
\mathcal M:L^2(0,1;\C^K)\longrightarrow L^2(0,1;\C^J),
\qquad
(\mathcal MF)(x)=M(x)F(x),
$$
where $M(x)=M_r$ on the positive-length cells $I_r$. Then the optimal squared lower and upper synthesis constants are
\begin{equation}\label{eq:v3-cell-synthesis-bounds}
A_E=\min_r\lambda_{\min}(M_r^*M_r),
\qquad
B_E=\max_r\lambda_{\max}(M_r^*M_r).
\end{equation}
The operator $\mathcal M$ is bounded below if and only if every $M_r$ has full column rank $K$.

The optimal squared lower constant for the adjoint is
\begin{equation}\label{eq:v3-cell-adjoint-lower}
A_{\mathrm{fr}}
=
\min_r\lambda_{\min}(M_rM_r^*),
\end{equation}
and it is positive if and only if every $M_r$ has full row rank $J$. The optimal upper constant for the adjoint is again $B_E$.
\end{lemma}

\begin{proof}
For $F\in L^2(0,1;\C^K)$,
$$
\norm{\mathcal MF}^2
=
\sum_r\int_{I_r}
ip{M_r^*M_rF(x)}{F(x)}\dd x.
$$
The finite-dimensional spectral theorem gives the inequalities with the constants in \eqref{eq:v3-cell-synthesis-bounds}. If $r_0$ attains the minimum and $u$ is a unit eigenvector of $M_{r_0}^*M_{r_0}$ corresponding to $A_E$, then
$$
F(x)=|I_{r_0}|^{-1/2}\mathbf1_{I_{r_0}}(x)u
$$
has norm one and satisfies $\norm{\mathcal MF}^2=A_E$. The analogous construction on a maximizing cell proves optimality of $B_E$.

The matrix $M_r^*M_r$ is positive definite exactly when $M_r$ has full column rank. Since only finitely many matrices occur, positivity on every cell implies a uniform positive minimum. Applying the same argument to multiplication by $M(x)^*$ gives \eqref{eq:v3-cell-adjoint-lower} and the row-rank criterion.
\end{proof}

\begin{remark}[Why finite range matters]\label{rem:v3-finite-range}
Full rank almost everywhere does not imply a uniform lower singular-value bound for an arbitrary measurable matrix field. Already in the scalar case $M(x)=x$ is nonzero for every $x\in(0,1)$, but
$$
F_n(x)=\sqrt n\,\mathbf1_{(0,1/n)}(x)
$$
satisfies $\norm{F_n}_{L^2}=1$ and
$$
\norm{M F_n}_{L^2}^2
=
n\int_0^{1/n}x^2\dd x
=
\frac1{3n^2}\longrightarrow0.
$$
The finite-cell property is precisely what converts cellwise rank into a uniform bound in the present setting.
\end{remark}

\begin{theorem}[Rank classification for translated unit intervals]\label{thm:v3-unit-rank}
The system generated by $\Lambda$ on $E$ has the following properties.
\begin{enumerate}[label=(\roman*)]
\item It is a Riesz sequence if and only if $J\ge K$ and every positive-cell matrix $M_r$ has column rank $K$.
\item It is a frame for $L^2(E)$ if and only if $K\ge J$ and every positive-cell matrix $M_r$ has row rank $J$.
\item It is a Riesz basis of $L^2(E)$ if and only if $J=K$ and every positive-cell matrix $M_r$ is invertible.
\end{enumerate}
The optimal squared synthesis bounds are the constants in \eqref{eq:v3-cell-synthesis-bounds}. In the frame case the optimal lower frame bound is \eqref{eq:v3-cell-adjoint-lower} and the optimal upper frame bound is $B_E$. In the basis case,
\begin{equation}\label{eq:v3-unit-inverse-norm}
\operatorname*{ess\,sup}_{x\in(0,1)}\norm{M(x)^{-1}}
=
\max_r\norm{M_r^{-1}}
=A_E^{-1/2}.
\end{equation}
\end{theorem}

\begin{proof}
Under the unitary coefficient transformation of Theorem~\ref{thm:v3-unit-fiberization}, synthesis is multiplication by $M(x)$. A Riesz sequence is equivalent to a positive lower synthesis bound and hence, by Lemma~\ref{lem:v3-finite-cell}, to full column rank on every cell. A frame is equivalent to a positive lower bound for the adjoint, hence to full row rank. A Riesz basis requires both properties and therefore forces $J=K$ and invertibility on every cell. Conversely, in the square case the finitely many inverses $M_r^{-1}$ define a bounded inverse multiplier. Formula \eqref{eq:v3-unit-inverse-norm} is immediate from
$$
\norm{M_r^{-1}}=\lambda_{\min}(M_r^*M_r)^{-1/2}.
$$
\end{proof}

\begin{proposition}[Finite-cell reduction and chamber rigidity]\label{prop:v3-unit-chambers}
For the unit-interval geometry, the matrix field assumes at most $J+1$ values. All optimal bounds and all rank criteria reduce to the positive-cell matrices $M_r$. Furthermore, if the integer parts $q_j$, the shifts $\delta_k$, and the weak ordering of the fractional parts $\beta_j$ are fixed, then the complete list of cell matrices, the optimal bounds, and the Riesz-sequence/frame/basis status remain unchanged while the $\beta_j$ vary without crossing a chamber boundary.
\end{proposition}

\begin{proof}
The vector
$$
\bigl(
\mathbf1_{[0,\beta_1)}(x),\ldots,
\mathbf1_{[0,\beta_J)}(x)
\bigr)
$$
changes only when $x$ crosses one of the fractional parts $\beta_j$. Thus the sequence of active subsets and hence the matrices $M_r$ depend only on the order in which the fractional parts are crossed. The lengths of the cells do not enter the extremal eigenvalues, because the optimal test functions can be localized to any positive-measure cell.
\end{proof}

\subsection{Arithmetic progressions of frequency shifts}

Suppose the shifted lattices form an arithmetic progression,
\begin{equation}\label{eq:v3-arithmetic-shifts}
\delta_k=\delta_0+(k-1)\tau,
\qquad 1\le k\le K.
\end{equation}
On a fixed cell the integers $\kappa_j$ are constant. Put
$$
z_j=\e^{2\pi i\tau\kappa_j}.
$$
Up to unitary diagonal factors, the local matrix becomes
$$
\begin{pmatrix}
1&z_1&\cdots&z_1^{K-1}\\
1&z_2&\cdots&z_2^{K-1}\\
\vdots&\vdots&&\vdots\\
1&z_J&\cdots&z_J^{K-1}
\end{pmatrix}.
$$

\begin{corollary}[Vandermonde basis criterion for unit intervals]\label{cor:v3-unit-vandermonde}
Assume $J=K$ and \eqref{eq:v3-arithmetic-shifts}. The exponential family is a Riesz basis if and only if, on every positive-measure cell,
\begin{equation}\label{eq:v3-unit-vandermonde-condition}
\tau(\kappa_j-\kappa_\ell)\notin\Z
\qquad(j\ne\ell).
\end{equation}
\end{corollary}

\begin{proof}
The determinant of the local Vandermonde factor is
$$
\prod_{1\le j<\ell\le J}(z_\ell-z_j),
$$
which is nonzero exactly when the nodes are pairwise distinct. This is equivalent to \eqref{eq:v3-unit-vandermonde-condition}. The result follows from Theorem~\ref{thm:v3-unit-rank}.
\end{proof}

For $J=K=2$, $\delta_1=0$, and $\delta_2=\delta$, this construction reproduces the constant matrix of Section~\ref{sec:integer-twointerval} and the two piecewise constant matrices of Section~\ref{sec:noninteger-twointerval}. Thus the earlier $2\times2$ formulas are exactly the two-channel specialization of the general unit-interval theory.

\subsection{Arbitrary finite unions and variable fiber dimensions}
\label{subsec:v3-unequal}

The common unit length is not essential to the folding argument. What changes for unequal lengths is the number of points of $E$ lying above a given fractional coordinate. This multiplicity must therefore be built into the target fiber.

The role of integer-translate multiplicity in exponential systems has been studied previously. \cite{DeCarliMizrahiTepper2019} characterized
frame and Riesz-sequence properties of the standard integer exponential system through the multiplicity
function
$$\Phi(x)=\sum_{m\in\mathbb Z}\mathbf{1}_E(x+m).$$
In the present setting, this multiplicity is $q(x)=|\Omega_x|$. For a single frequency lattice, the matrix formulation below reduces to the corresponding scalar fiber description. With several shifted lattices, the multiplicity function alone no longer determines stability: the ranks and singular values of the associated phase matrices must also be taken into account.

Let
\begin{equation}\label{eq:v3-general-E}
E=\bigcup_{j=1}^J[a_j,b_j),
\qquad a_j<b_j,
\end{equation}
be a finite union of bounded intervals that are pairwise disjoint up to endpoints. Fix $K\ge1$ and pairwise incongruent shifts $\delta_1,\ldots,\delta_K$ modulo $\Z$. Let
\begin{equation}\label{eq:v3-general-Lambda}
\Lambda=\bigcup_{k=1}^K(\Z+\delta_k),
\end{equation}
and define synthesis by
\begin{equation}\label{eq:v3-general-synthesis}
(Sa)(y)
=
\sum_{k=1}^K\sum_{n\in\Z}
a_n^{(k)}\e^{2\pi i(n+\delta_k)y}.
\end{equation}

For $x\in[0,1)$ define the integer fiber
\begin{equation}\label{eq:v3-Omegax}
\Omega_x=\{m\in\Z:x+m\in E\},
\qquad
q(x)=|\Omega_x|.
\end{equation}
Because $E$ is bounded, every $\Omega_x$ is contained in one fixed finite subset of $\Z$.

Partition $(0,1)$ at the distinct points
\begin{equation}\label{eq:v3-general-partition}
\{0,1\}\cup
\bigl\{\{a_j\},\{b_j\}:1\le j\le J\bigr\},
\end{equation}
where $\{t\}=t-\lfloor t\rfloor$ denotes fractional part. Let the resulting positive-length open cells be $I_1,\ldots,I_R$. Membership of $x+m$ in $E$ can change only at one of the partition points, so
\begin{equation}\label{eq:v3-Omega-cell}
\Omega_x=\Omega_r,
\qquad
q(x)=q_r:=|\Omega_r|,
\qquad
x\in I_r.
\end{equation}
Some $q_r$ may be zero.

\begin{proposition}[Folding with variable fiber dimensions]\label{prop:v3-variable-folding}
With the convention $\C^0=\{0\}$, define
\begin{equation}\label{eq:v3-HE}
\mathcal H_E
=
\bigoplus_{r=1}^R L^2(I_r;\C^{q_r}),
\end{equation}
and
\begin{equation}\label{eq:v3-W}
(Wf)_r(x)
=
\bigl(f(x+m)\bigr)_{m\in\Omega_r}.
\end{equation}
Then
$$
W:L^2(E)\longrightarrow\mathcal H_E
$$
is unitary.

Define
\begin{equation}\label{eq:v3-general-Ak}
A_k(x)
=
\e^{2\pi i\delta_kx}
\sum_{n\in\Z}a_n^{(k)}\e^{2\pi inx},
\qquad
A(x)=(A_1(x),\ldots,A_K(x))^T.
\end{equation}
The coefficient map is unitary from $\bigoplus_{k=1}^K\ell^2(\Z)$ onto $L^2(0,1;\C^K)$. Under these transformations,
\begin{equation}\label{eq:v3-variable-matrix-action}
(WSa)_r(x)=M_rA(x),
\qquad
M_r
=
\left(
\e^{2\pi i\delta_km}
\right)_{
\substack{m\in\Omega_r\\1\le k\le K}
}
\in\C^{q_r\times K}.
\end{equation}
Consequently,
\begin{equation}\label{eq:v3-variable-norm-identity}
\norm{Sa}_{L^2(E)}^2
=
\sum_{r=1}^R\int_{I_r}\norm{M_rA(x)}_{\C^{q_r}}^2\dd x,
\end{equation}
while
\begin{equation}\label{eq:v3-variable-coeff-norm}
\sum_{k=1}^K\norm{a^{(k)}}_{\ell^2}^2
=
\int_0^1\norm{A(x)}_{\C^K}^2\dd x.
\end{equation}
\end{proposition}

\begin{proof}
Decompose $\R$ into the unit cells $[m,m+1)$ and use $y=x+m$. Then
\begin{align*}
\int_E|f(y)|^2\dd y
&=
\int_0^1\sum_{m\in\Omega_x}|f(x+m)|^2\dd x\\
&=
\sum_{r=1}^R\int_{I_r}
\norm{(Wf)_r(x)}_{\C^{q_r}}^2\dd x.
\end{align*}
Thus $W$ is an isometry. It is onto because a vector field $(F_r)_r\in\mathcal H_E$ defines a function on $E$ by
$$
f(x+m)=(F_r(x))_m,
\qquad x\in I_r,
\quad m\in\Omega_r.
$$
The representation $y=x+m$ with $x\in[0,1)$ is unique outside a null set of endpoints.

Parseval shows that
$$
a^{(k)}\longmapsto\sum_na_n^{(k)}\e^{2\pi inx}
$$
is unitary, and multiplication by $\e^{2\pi i\delta_kx}$ is unitary, proving the coefficient statement.

For a finite coefficient family and $m\in\Z$,
\begin{align*}
(Sa)(x+m)
&=
\sum_{k=1}^K
\e^{2\pi i\delta_km}
\e^{2\pi i\delta_kx}
\sum_{n\in\Z}a_n^{(k)}\e^{2\pi in(x+m)}\\
&=
\sum_{k=1}^K
\e^{2\pi i\delta_km}A_k(x),
\end{align*}
using $\e^{2\pi inm}=1$. This is exactly \eqref{eq:v3-variable-matrix-action}. Since only finitely many matrices occur, the multiplier is bounded, and density extends the identities to all square-summable coefficient families.
\end{proof}

\begin{theorem}[Exact rank criteria for arbitrary finite unions]\label{thm:v3-variable-rank}
The exponential family associated with \eqref{eq:v3-general-Lambda} on the set $E$ in \eqref{eq:v3-general-E} satisfies the following equivalences.
\begin{enumerate}[label=(\roman*)]
\item It is a Riesz sequence in $L^2(E)$ if and only if every $M_r$ has column rank $K$. Necessarily $q_r\ge K$ on every cell.
\item It is a frame for $L^2(E)$ if and only if every $M_r$ with $q_r>0$ has row rank $q_r$. Necessarily $q_r\le K$ on every nonempty cell.
\item It is a Riesz basis of $L^2(E)$ if and only if
$$
q_r=K
$$
and $M_r$ is invertible for every cell.
\end{enumerate}
The optimal squared synthesis bounds are
\begin{equation}\label{eq:v3-variable-seq-bounds}
A_{\mathrm{seq}}
=
\min_r\lambda_{\min}(M_r^*M_r),
\qquad
B
=
\max_r\lambda_{\max}(M_r^*M_r),
\end{equation}
with the convention $M_r^*M_r=0_{K\times K}$ if $q_r=0$. In the frame case, the optimal lower frame bound is
\begin{equation}\label{eq:v3-variable-frame-lower}
A_{\mathrm{fr}}
=
\min_{r:q_r>0}
\lambda_{\min}(M_rM_r^*),
\end{equation}
and the optimal upper frame bound is $B$.
\end{theorem}

\begin{proof}
By Proposition~\ref{prop:v3-variable-folding}, synthesis acts on $I_r$ as the finite matrix $M_r$, while the adjoint acts as $M_r^*$. A positive lower synthesis bound is equivalent to positivity of every $\lambda_{\min}(M_r^*M_r)$ and therefore to full column rank on every cell. A lower frame bound is equivalent to positivity of $\lambda_{\min}(M_rM_r^*)$ on every nonempty target cell and hence to full row rank.

The exact constants follow from the same localization argument as in Lemma~\ref{lem:v3-finite-cell}: choose an extremizing cell and an extremal unit eigenvector, then multiply that vector by the normalized indicator of the cell. Because only finitely many cells occur, the minimum of the positive local singular values is positive whenever all required ranks are full.

A Riesz basis requires both the sequence and frame properties. Thus every cell must satisfy simultaneously $q_r\ge K$ and $q_r\le K$, so $q_r=K$ and $M_r$ must be invertible. Conversely, if every $q_r=K$ and every $M_r$ is invertible, the cellwise inverse multiplier is bounded because there are only finitely many inverse matrices.
\end{proof}

\begin{remark}[Geometric multiplicity condition]\label{rem:v3-multitiling}
For a Riesz basis with $K$ shifted integer lattices, Theorem~\ref{thm:v3-variable-rank} forces
\begin{equation}\label{eq:v3-multiplicity-K}
\sum_{m\in\Z}\mathbf1_E(x+m)=K
\qquad\text{for almost every }x\in(0,1).
\end{equation}
Integrating gives the necessary measure identity
$$
|E|=K.
$$
The converse is false: total measure $K$ does not imply constant fiber multiplicity, and constant multiplicity by itself still has to be accompanied by invertibility of the prescribed phase matrices.
\end{remark}

\begin{proposition}[Chamber rigidity for arbitrary finite unions]\label{prop:v3-general-chambers}
Fix the integer parts of all endpoints $a_j,b_j$, the shifts $\delta_k$, and the combinatorial order of the fractional endpoint set in \eqref{eq:v3-general-partition}. As the endpoint fractional parts vary within a chamber without crossing one another or an integer boundary, the ordered list of integer fibers $\Omega_r$ remains unchanged. Consequently, the matrices $M_r$, all exact bounds in Theorem~\ref{thm:v3-variable-rank}, and the Riesz-sequence/frame/basis status remain unchanged; only the lengths of the cells vary.
\end{proposition}

\begin{proof}
For each fixed integer $m$, membership of $x+m$ in any interval $[a_j,b_j)$ can change only when $x$ crosses the fractional part of an endpoint. If the integer parts and the order of these crossing events are fixed, the same membership changes occur in the same order. Hence the fiber sets $\Omega_r$ and therefore the phase matrices $M_r$ are unchanged.
\end{proof}

\subsection{Examples with unequal interval lengths}

\begin{example}[Unequal lengths with an exact Riesz basis]\label{ex:v3-unequal-basis}
Take
$$
E=[0,1/2)\cup[3/2,3),
\qquad
\delta_1=0,
\qquad
\delta_2=1/3.
$$
The component lengths are $1/2$ and $3/2$. The integer fibers are
$$
\Omega_x
=
\begin{cases}
\{0,2\},&0<x<1/2,\\
\{1,2\},&1/2<x<1.
\end{cases}
$$
Let $\zeta=\e^{2\pi i/3}$. The two local matrices are
$$
M_1=
\begin{pmatrix}
1&1\\
1&\zeta^2
\end{pmatrix},
\qquad
M_2=
\begin{pmatrix}
1&\zeta\\
1&\zeta^2
\end{pmatrix}.
$$
Each Gram matrix has diagonal entries $2$ and an off-diagonal entry of modulus $1$. Hence each has eigenvalues $1$ and $3$. Therefore
$$
\{\e^{2\pi iny}:n\in\Z\}
\cup
\{\e^{2\pi i(n+1/3)y}:n\in\Z\}
$$
is a Riesz basis of $L^2(E)$ with exact squared bounds
$$
A=1,
\qquad
B=3.
$$
\end{example}

\begin{example}[Total measure alone is insufficient]\label{ex:v3-measure-insufficient}
Let
$$
E=[0,1/2)\cup[1,5/2).
$$
Again $|E|=2$, but
$$
q(x)
=
\begin{cases}
3,&0<x<1/2,\\
1,&1/2<x<1.
\end{cases}
$$
For any choice of two shifted integer lattices, the matrix on the first cell has three rows and two columns and therefore cannot have full row rank; the matrix on the second cell has one row and two columns and therefore cannot have full column rank. The corresponding two-lattice exponential family is consequently neither a frame nor a Riesz sequence in $L^2(E)$.
\end{example}

\subsection{Arithmetic-shift criteria in the variable-fiber setting}

When the shifts have the arithmetic form \eqref{eq:v3-arithmetic-shifts}, each cell matrix in \eqref{eq:v3-variable-matrix-action} has rows indexed by integers $m\in\Omega_r$ and columns
$$
\e^{2\pi i(\delta_0+(k-1)\tau)m}
=
\e^{2\pi i\delta_0m}
\left(\e^{2\pi i\tau m}\right)^{k-1}.
$$
Thus, up to a unitary diagonal row factor, $M_r$ is a rectangular Vandermonde matrix with nodes
$$
z_m=\e^{2\pi i\tau m},
\qquad m\in\Omega_r.
$$
In the basis case $q_r=K$, Theorem~\ref{thm:v3-variable-rank} immediately gives the following criterion.

\begin{corollary}[Vandermonde criterion on arbitrary finite unions]\label{cor:v3-general-vandermonde}
Assume the shifts satisfy \eqref{eq:v3-arithmetic-shifts}. The corresponding family is a Riesz basis of $L^2(E)$ if and only if
\begin{enumerate}[label=(\roman*)]
\item $q_r=K$ for every cell, and
\item for every cell and every two distinct integers $m,m'\in\Omega_r$,
\begin{equation}\label{eq:v3-general-vandermonde-condition}
\tau(m-m')\notin\Z.
\end{equation}
\end{enumerate}
\end{corollary}

\begin{proof}
On a cell with $q_r=K$, the determinant is a nonzero unimodular factor times
$$
\prod_{\substack{m<m'\\ m,m'\in\Omega_r}}
\left(\e^{2\pi i\tau m'}-\e^{2\pi i\tau m}\right).
$$
It is nonzero precisely under \eqref{eq:v3-general-vandermonde-condition}.
\end{proof}

\section{Applications of the finite-cell framework}
\label{sec:finite-cell-applications}

The finite-cell formulation developed in Section~\ref{sec:finite-unions}
has two further applications. We first use periodicity of spectra for finite
unions of intervals to obtain finite Hadamard constraints and arithmetic
rigidity on the integer fibers. We then show that rationally commensurate
multirate sampling systems reduce exactly to finitely many common-lattice
channels governed by the same phase matrices.

\subsection{Periodic spectra, Hadamard fibers, and spectral rigidity}
\label{subsec:fuglede-fibers}

We next connect the variable-fiber representation of Section~\ref{sec:finite-unions} with spectral sets. Let $\Omega\subset\R$ be a bounded measurable set of measure one. Recall that $\Lambda\subset\R$ is a spectrum for $\Omega$ if the exponentials
$$
\{\e^{2\pi i\lambda x}:\lambda\in\Lambda\}
$$
form an orthogonal basis of $L^2(\Omega)$. For finite unions of intervals, every spectrum is periodic~\cite{BoseMadan2011}. Recent work has established the spectral-implies-tiling direction of Fuglede's conjecture for unions of three intervals~\cite{Shi2026ThreeIntervals}, while the corresponding general problem for four or more intervals remains a natural next case.

Assume that a spectrum of $\Omega$ has period $d\in\mathbb N$ and is written as
\begin{equation}\label{eq:periodic-spectrum-form}
\Lambda
=
\bigcup_{k=1}^{d}(d\Z+\lambda_k),
\qquad
\lambda_k\not\equiv\lambda_\ell\pmod d
\quad(k\neq\ell).
\end{equation}
For a normalized spectral set, the number of representatives in a period equals $d$ by the density theorem. Define
$$
E=d\Omega,
\qquad
\delta_k=\lambda_k/d.
$$
Then
$$
\Lambda/d
=
\bigcup_{k=1}^{d}(\Z+\delta_k).
$$
The fact that a Riesz basis formed by finitely many translates of a lattice forces a corresponding multi-tiling condition is known; see \cite{AgoraAntezanaCabrelli2015}. In the orthonormal spectral setting considered here, the finite-cell representation additionally forces each local square phase matrix to be a scaled unitary matrix. We use these Hadamard identities to investigate the arithmetic constraints generated by transitions between adjacent fibers.
\begin{proposition}[Finite Hadamard fibers associated with a periodic spectrum]
\label{prop:spectral-hadamard-fibers}
Let $\Omega$ be a finite union of bounded intervals of measure one and let \eqref{eq:periodic-spectrum-form} be a spectrum. For $E=d\Omega$, let
$$
\Omega_x^E
=
\{m\in\Z:x+m\in E\},
\qquad 0\leq x<1,
$$
and let $\Omega_r^E$ be its finitely many cell values as in Proposition~\ref{prop:v3-variable-folding}. Then
\begin{equation}\label{eq:spectral-fiber-cardinality}
|\Omega_r^E|=d
\end{equation}
for every positive-measure cell, and the phase matrix
\begin{equation}\label{eq:spectral-fiber-matrix}
M_r
=
\left(
\e^{2\pi i\delta_k m}
\right)_{
\substack{m\in\Omega_r^E\\1\leq k\leq d}}
\end{equation}
satisfies
\begin{equation}\label{eq:hadamard-fiber}
M_r^*M_r=dI_d.
\end{equation}
In particular, $d^{-1/2}M_r$ is unitary for every cell.
\end{proposition}

\begin{proof}
Under the change of variables $y=dx$, the spectral family on $\Omega$ becomes, up to the common spatial scaling, the family with frequency set $\Lambda/d$ on $E=d\Omega$. Since the original exponentials are orthogonal and $|E|=d$, the scaled synthesis operator satisfies
$$
\norm{Sa}_{L^2(E)}^2
=d\sum_{k=1}^{d}\norm{a^{(k)}}_2^2.
$$
Apply Proposition~\ref{prop:v3-variable-folding}. The coefficient map onto $L^2(0,1;\C^d)$ is unitary, while on a cell $I_r$ synthesis is multiplication by $M_r$. Localizing an arbitrary vector function to $I_r$ gives
$$
\norm{M_rz}_{\C^{|\Omega_r^E|}}^2
=d\norm z_{\C^d}^2
\qquad(z\in\C^d).
$$
Hence $M_r^*M_r=dI_d$, so $M_r$ has rank $d$ and $|\Omega_r^E|\geq d$.

Since
$$
\int_0^1|\Omega_x^E|\,dx=|E|=d,
$$
and $|\Omega_x^E|\geq d$ almost everywhere, equality must hold almost everywhere. This proves \eqref{eq:spectral-fiber-cardinality} and \eqref{eq:hadamard-fiber}.
\end{proof}

The preceding proposition has an elementary but useful geometric consequence. Let $c\in[0,1)$ be a fractional endpoint of $E$. As $x$ crosses $c$, the cardinality of $\Omega_x^E$ changes by the number of interval left endpoints having fractional part $c$ minus the number of right endpoints having fractional part $c$. Since the cardinality is constant, these multiplicities must balance.

\begin{lemma}[Endpoint balance]
\label{lem:endpoint-balance}
Under the assumptions of Proposition~\ref{prop:spectral-hadamard-fibers}, for every $c\in[0,1)$ the number of left endpoints of $E$ with fractional part $c$ equals the number of right endpoints of $E$ with fractional part $c$, counted with multiplicity.
\end{lemma}

\begin{proof}
The jump of $x\mapsto|\Omega_x^E|$ at $c$ is exactly the difference between these two multiplicities. Proposition~\ref{prop:spectral-hadamard-fibers} shows that this function equals $d$ almost everywhere, hence every jump is zero.
\end{proof}

Thus adjacent fibers have equal cardinality and differ by exchanges. The next lemma is the basic rigidity mechanism.

\begin{lemma}[Single-row exchange rigidity]
\label{lem:single-row-exchange}
Let $\Gamma,\Gamma'\subset\Z$ have cardinality $d$ and suppose
$$
\Gamma'=(\Gamma\setminus\{a\})\cup\{b\},
\qquad a\neq b.
$$
Let $D=\{\delta_1,\ldots,\delta_d\}\subset\R/\Z$, and assume that both phase matrices
$$
M(\Gamma,D)
=
\left(\e^{2\pi i\delta_km}\right)_{m\in\Gamma,\,k},
\qquad
M(\Gamma',D)
=
\left(\e^{2\pi i\delta_km}\right)_{m\in\Gamma',\,k}
$$
satisfy $M^*M=dI_d$. Then
\begin{equation}\label{eq:single-exchange-arithmetic}
(\delta_j-\delta_k)(b-a)\in\Z
\qquad
(1\leq j,k\leq d).
\end{equation}
\end{lemma}

\begin{proof}
For $j\neq k$, orthogonality of columns gives
$$
\sum_{m\in\Gamma}
\e^{2\pi i(\delta_j-\delta_k)m}=0,
\qquad
\sum_{m\in\Gamma'}
\e^{2\pi i(\delta_j-\delta_k)m}=0.
$$
Subtracting the two equalities yields
$$
\e^{2\pi i(\delta_j-\delta_k)b}
=
\e^{2\pi i(\delta_j-\delta_k)a},
$$
which is equivalent to \eqref{eq:single-exchange-arithmetic}. The case $j=k$ is trivial.
\end{proof}

More generally, suppose an adjacent-cell transition removes $A=\{a_1,\ldots,a_r\}$ and inserts $B=\{b_1,\ldots,b_r\}$. The same subtraction gives the following exact identity.

\begin{lemma}[Rank-$r$ exchange identity]
\label{lem:rank-r-exchange}
Under the same Hadamard hypotheses, if
$$
\Gamma'=(\Gamma\setminus A)\cup B,
\qquad |A|=|B|=r,
$$
then, for every $j\neq k$,
\begin{equation}\label{eq:rank-r-exchange}
\sum_{\nu=1}^{r}
\e^{2\pi i(\delta_j-\delta_k)a_\nu}
=
\sum_{\nu=1}^{r}
\e^{2\pi i(\delta_j-\delta_k)b_\nu}.
\end{equation}
\end{lemma}

For two-row exchanges one obtains a useful dichotomy.

\begin{lemma}[Two-row exchange dichotomy]
\label{lem:two-row-exchange}
Assume $r=2$ in Lemma~\ref{lem:rank-r-exchange} and fix $\theta=\delta_j-\delta_k$. Then either
\begin{equation}\label{eq:two-row-permutation}
\left\{
\e^{2\pi i\theta a_1},
\e^{2\pi i\theta a_2}
\right\}
=
\left\{
\e^{2\pi i\theta b_1},
\e^{2\pi i\theta b_2}
\right\},
\end{equation}
or both pairs are antipodal:
\begin{equation}\label{eq:two-row-antipodal}
\e^{2\pi i\theta(a_1-a_2)}=-1,
\qquad
\e^{2\pi i\theta(b_1-b_2)}=-1.
\end{equation}
\end{lemma}

\begin{proof}
Put $z_\nu=\e^{2\pi i\theta a_\nu}$ and $w_\nu=\e^{2\pi i\theta b_\nu}$. Equation~\eqref{eq:rank-r-exchange} says $z_1+z_2=w_1+w_2=:s$. If $s=0$, each pair consists of antipodal points and \eqref{eq:two-row-antipodal} follows. If $s\neq0$, then for any two unit-modulus numbers with sum $s$,
$$
z_1z_2=\frac{s}{\overline s}.
$$
Thus $z_1,z_2$ and $w_1,w_2$ are the roots of the same quadratic polynomial
$$
t^2-st+\frac{s}{\overline s},
$$
which proves \eqref{eq:two-row-permutation}.
\end{proof}

The single-exchange lemma already yields a conditional tiling criterion. Consider the cyclic sequence of positive-measure cells and assume every transition is a single-row exchange. Write the corresponding integer displacements as
$$
q_1,\ldots,q_R,
\qquad
q_r=b_r-a_r,
$$
and put
\begin{equation}\label{eq:transition-gcd}
g=\gcd(|q_1|,\ldots,|q_R|).
\end{equation}

\begin{theorem}[A fiber-transition criterion forcing lattice spectrality and tiling]
\label{thm:fiber-transition-tiling}
Under the assumptions of Proposition~\ref{prop:spectral-hadamard-fibers}, suppose every adjacent fiber transition is a single-row exchange. Then
\begin{equation}\label{eq:D-difference-group}
D-D
\subset
\frac1g\Z/\Z,
\qquad
D=\{\delta_1,\ldots,\delta_d\}.
\end{equation}
In particular $g\geq d$. If $g=d$, then $D$ is a translate of the full subgroup $\frac1d\Z/\Z$, the original spectrum $\Lambda$ is a translate of $\Z$, and $\Omega$ tiles $\R$ by translations of $\Z$.
\end{theorem}

\begin{proof}
Lemma~\ref{lem:single-row-exchange} gives
$$
(\delta_j-\delta_k)q_r\in\Z
$$
for every $r,j,k$. Bezout's identity applied to the integers $q_r$ gives
$$
g(\delta_j-\delta_k)\in\Z,
$$
which proves \eqref{eq:D-difference-group}. Hence, after subtracting one element of $D$, the set $D$ is contained in the subgroup $\frac1g\Z/\Z$, which has $g$ elements. Since $D$ contains $d$ distinct elements, $g\geq d$.

If $g=d$, cardinality forces
$$
D
=\delta_0+\frac1d\{0,1,\ldots,d-1\}
\pmod1.
$$
Therefore
$$
\Lambda/d
=D+\Z
=\delta_0+\frac1d\Z,
$$
and hence
$$
\Lambda=d\delta_0+\Z.
$$
Modulation by $\e^{2\pi i d\delta_0x}$ shows that $\Z$ itself is a spectrum for $\Omega$. Let
$$
P(x)=\sum_{n\in\Z}\mathbf1_\Omega(x+n).
$$
The nonzero Fourier coefficients of the one-periodic function $P$ are
$$
\widehat P(k)
=
\int_\Omega \e^{-2\pi ikx}\,dx
=0,
\qquad k\in\Z\setminus\{0\},
$$
by orthogonality, while $\widehat P(0)=|\Omega|=1$. Hence $P=1$ almost everywhere, so $\Omega$ tiles $\R$ by $\Z$.
\end{proof}

For a union of four intervals, Proposition~\ref{prop:spectral-hadamard-fibers} and Lemma~\ref{lem:endpoint-balance} reduce the geometry to finitely many balanced endpoint transitions. Transitions of multiplicity one are completely controlled by Lemma~\ref{lem:single-row-exchange}, and transitions of multiplicity two satisfy Lemma~\ref{lem:two-row-exchange}. The remaining obstruction is therefore finite and algebraic: one must classify the simultaneous identities \eqref{eq:rank-r-exchange} for the finitely many differences $D-D$, including the higher-multiplicity transitions that can occur when several endpoints have the same fractional part. This provides a possible finite-matrix route to the four-interval case of the one-dimensional Fuglede problem, but no general classification is asserted here.

\subsection{Commensurate multirate sampling}
\label{subsec:commensurate-multirate}

We finally show that the variable-fiber theorem applies without change to finite unions of sampling lattices having different but rationally commensurate densities. This gives a partial framework for the multirate reconstruction problem discussed, for example, by Frederick and Yacoubou Djima~\cite{FrederickYacoubouDjima2024}, who ask for extensions to more than two sampling sets of different densities.

Let
\begin{equation}\label{eq:multirate-frequency-set}
\Lambda
=
\bigcup_{j=1}^{J}
(\alpha_j\Z+\beta_j),
\qquad \alpha_j>0,
\end{equation}
and assume
\begin{equation}\label{eq:commensurate-steps}
\frac{\alpha_j}{\alpha_\ell}\in\mathbb Q
\qquad(1\leq j,\ell\leq J).
\end{equation}
Choose $L>0$ such that
\begin{equation}\label{eq:common-multiple-step}
q_j:=\frac{L}{\alpha_j}\in\mathbb N
\qquad(1\leq j\leq J).
\end{equation}
Then
\begin{equation}\label{eq:lattice-coset-decomposition}
\alpha_j\Z+\beta_j
=
\bigcup_{r=0}^{q_j-1}
\left(L\Z+\beta_j+r\alpha_j\right).
\end{equation}
After merging repetitions modulo $L$, let
$$
\gamma_1,\ldots,\gamma_K
$$
be the distinct coset representatives and set
\begin{equation}\label{eq:multirate-deltas}
\delta_k=\gamma_k/L\pmod1.
\end{equation}

\begin{theorem}[Exact reduction of commensurate multirate lattices]
\label{thm:commensurate-multirate}
Let $E\subset\R$ be a finite union of bounded intervals and let $\Lambda$ be given by \eqref{eq:multirate-frequency-set}--\eqref{eq:commensurate-steps}. Put
$$
\widetilde E=LE.
$$
Then the exponential system indexed by $\Lambda$ in $L^2(E)$ has the same Riesz-sequence, frame, and Riesz-basis properties as the system indexed by
\begin{equation}\label{eq:scaled-common-channel-set}
\widetilde\Lambda
=
\bigcup_{k=1}^{K}(\Z+\delta_k)
\end{equation}
in $L^2(\widetilde E)$.

Let $\widetilde\Omega_r$ be the finitely many integer fibers of $\widetilde E$ and define
\begin{equation}\label{eq:multirate-fiber-matrix}
M_r
=
\left(
\e^{2\pi i\delta_km}
\right)_{
\substack{m\in\widetilde\Omega_r\\1\leq k\leq K}}.
\end{equation}
Then:
\begin{enumerate}[label=(\roman*)]
\item the original system is a Riesz sequence if and only if every $M_r$ has column rank $K$;
\item it is a frame if and only if every nonempty $M_r$ has full row rank;
\item it is a Riesz basis if and only if $|\widetilde\Omega_r|=K$ and $M_r$ is invertible on every cell.
\end{enumerate}
If $\widetilde A$ and $\widetilde B$ denote the optimal squared bounds computed from the singular values of the matrices $M_r$ as in Theorem~\ref{thm:v3-variable-rank}, then the corresponding squared synthesis bounds on the original domain $E$ are
\begin{equation}\label{eq:multirate-scaled-bounds}
A=\frac{\widetilde A}{L},
\qquad
B=\frac{\widetilde B}{L}.
\end{equation}
\end{theorem}

\begin{proof}
Equation~\eqref{eq:lattice-coset-decomposition} rewrites $\Lambda$ as a finite union of cosets of the common lattice $L\Z$. After repetitions are removed,
$$
\Lambda
=
\bigcup_{k=1}^{K}(L\Z+\gamma_k).
$$
With $y=Lx$,
$$
\e^{2\pi i(Ln+\gamma_k)x}
=
\e^{2\pi i(n+\delta_k)y}.
$$
Moreover,
$$
\int_E|f(x)|^2\,dx
=
\frac1L\int_{\widetilde E}|f(y/L)|^2\,dy.
$$
Thus the synthesis inequality on $E$ is exactly the synthesis inequality for \eqref{eq:scaled-common-channel-set} on $\widetilde E$, multiplied by $1/L$. The rank criteria and the optimal scaled bounds now follow directly from Theorem~\ref{thm:v3-variable-rank}.
\end{proof}

The Fourier-dual relation between exponential frames and stable sampling is classical; see, for example, \cite{OrtegaCerdaSeip2002,AldroubiGrochenig2001}. The reconstruction formula below concerns the inverse of the exponential synthesis representation. Its interpretation as a sampling formula for bandlimited signals requires the corresponding Fourier-dual identification.

The reduction also gives an explicit reconstruction map whenever the system is a Riesz basis. 
\begin{proposition}[Fiberwise reconstruction in the commensurate multirate case]
\label{prop:commensurate-reconstruction}
Assume the hypotheses of Theorem~\ref{thm:commensurate-multirate} and suppose the equivalent Riesz-basis conditions hold. Let $f\in L^2(E)$ and put
$$
\widetilde f(y)=f(y/L),
\qquad y\in\widetilde E.
$$
On a cell $I_r$, define the folded data vector
$$
F_r(x)
=
\left(
\widetilde f(x+m)
\right)_{m\in\widetilde\Omega_r}.
$$
Then the channel functions are recovered pointwise by
\begin{equation}\label{eq:multirate-pointwise-inverse}
A(x)=M_r^{-1}F_r(x),
\qquad x\in I_r.
\end{equation}
Writing
$$
B_k(x)=\e^{-2\pi i\delta_kx}A_k(x),
$$
the coefficient sequence in the $k$-th common-lattice channel is the ordinary Fourier coefficient sequence
\begin{equation}\label{eq:multirate-fourier-coefficients}
a_n^{(k)}
=
\int_0^1 B_k(x)\e^{-2\pi inx}\,dx.
\end{equation}
After identifying each common-lattice channel with the corresponding coset in \eqref{eq:lattice-coset-decomposition}, these coefficients give the unique expansion of $f$ in the original multirate exponential basis.
\end{proposition}

\begin{proof}
This is the inverse of the fiberization in Proposition~\ref{prop:v3-variable-folding}. In the Riesz-basis case every $M_r$ is square and invertible, giving \eqref{eq:multirate-pointwise-inverse}. The modulation used in \eqref{eq:v3-general-Ak} is then removed, and ordinary Fourier inversion on $[0,1)$ gives \eqref{eq:multirate-fourier-coefficients}.
\end{proof}

Theorem~\ref{thm:commensurate-multirate} does not address incommensurate lattice steps or genuinely non-lattice sampling sets. It does, however, show that any finite collection of rationally commensurate sampling densities is already contained in the finite-cell matrix-symbol framework developed above. In this regime the existence question, the optimal condition number, and the reconstruction map all reduce to finitely many explicit matrices.

\section{Discussion and open problems}
\label{sec:discussion}
The results above identify two regimes in which the infinite-dimensional structure can be reduced exactly to finite-dimensional data. For periodic local perturbations, the synthesis operator is governed by a single constant matrix. For finite unions of intervals and finitely many shifted lattices, folding modulo one produces finitely many phase matrices, possibly of different row dimensions. Outside these structured regimes, the relevant invariants remain genuinely infinite-dimensional.

Several problems therefore remain open. For a general nonperiodic binary pattern, it would be useful to characterize geometrically the exact invariant
$$
\rho_{\delta,s}=\|P_-T_{2\delta}P_+\|.
$$
At the critical value $\delta=1/4$, the periodic approximation theorem reduces the worst-case lower bound to a sequence of finite extremal problems, but both the balanced two-block extremizer conjecture and the asymptotic decay of $A_p(1/4)$ remain open. A related question is whether the half-shifted Hilbert model can be compared quantitatively with the periodized cross blocks.

Although support-independent operator-norm perturbation estimates are available, a positive near-critical stability criterion for arbitrary nonperiodic sign configurations remains open. In the finite-alphabet setting, the norm-matrix criterion is only sufficient in the nonperiodic case, and an exact spectral classification remains open.

On disconnected domains, the finite-cell theory suggests two natural extensions. The first is a classification of the higher-multiplicity endpoint exchanges that arise in the four-interval spectral problem. The second is the extension of the commensurate multirate reduction to incommensurate lattice steps and genuinely non-lattice sampling sets.

\section{Conclusion}\label{sec:conclusion}

The main theme of the paper is that shifted-sinc operators provide a common language for several exponential-system problems, while the relevant finite-dimensional reductions depend strongly on the underlying structure. For a fixed binary local perturbation, the exact squared synthesis bounds are determined by the single compression norm
$$
\rho_{\delta,s}=\|P_-T_{2\delta}P_+\|.
$$
Periodicity turns this infinite-dimensional invariant into a constant finite matrix problem, and the worst-case binary bound can in turn be approximated by periodic patterns.

For disconnected spatial domains, the corresponding reduction is geometric rather than pattern-based. Folding modulo one yields finitely many phase matrices indexed by the integer fibers of the domain. Their singular values and ranks provide exact quantitative bounds and complete Riesz-sequence, frame, and basis criteria, including the case of unequal interval lengths.

These two reductions also lead to applications beyond the original models: periodic spectra produce finite Hadamard fibers and arithmetic rigidity, while rationally commensurate multirate systems reduce to common-lattice channels. The remaining difficulties are concentrated in the genuinely nonperiodic, incommensurate, or asymptotic regimes, where no single finite matrix captures the entire problem.

\section*{Acknowledgments}
P. Vellucci is a member of the INdAM Research Group GNCS.

\bibliographystyle{model1b-num-names}
\bibliography{example}
\end{document}